\documentclass[12pt, reqno]{amsart}
\usepackage{amssymb,amsfonts,amsmath,amsopn,amstext,amscd,color,latexsym,mathrsfs, stmaryrd}
\usepackage[shortlabels]{enumitem}
\usepackage{mathabx, colonequals, comment}
\usepackage[utf8]{inputenc}
\usepackage[all, cmtip]{xy}
\usepackage{mathtools}
\usepackage{enumitem}
\usepackage{tikz}
\usepackage{appendix}
\usetikzlibrary{cd}
\usetikzlibrary{decorations.pathmorphing,arrows,positioning}

\usepackage{microtype}
\usepackage{subfiles}

\usepackage[colorlinks, pagebackref]{hyperref}
\hypersetup{
colorlinks=true,
citecolor=red,
linkcolor=blue,
urlcolor=cyan}

\theoremstyle{plain}

\newtheorem{thm}{\bf Theorem}[subsection]
\newtheorem{lem}[thm]{\bf Lemma}
\newtheorem{cor}[thm]{\bf Corollary}
\newtheorem{prop}[thm]{\bf Proposition}

\theoremstyle{definition}
\newtheorem{nota}[thm]{\bf Notations}

\newtheorem{rem}[thm]{\bf Remark}

\theoremstyle{definition}

\newtheorem{defn}[thm]{\bf Definition}

\usepackage{chngcntr}

\newcommand{\bbA}{{\mathbb A}}

\newcommand{\bbC}{{\mathbb C}}

\newcommand{\bbF}{{\mathbb F}}

\newcommand{\bbN}{{\mathbb N}}

\newcommand{\bbP}{{\mathbb P}}
\newcommand{\bbQ}{{\mathbb Q}}

\newcommand{\bbZ}{{\mathbb Z}}

\newcommand{\bG}{{\bf G}}

\newcommand{\cA}{{\mathcal A}}

\newcommand{\cD}{{\mathcal D}}

\newcommand{\ccH}{{\mathcal H}}

\newcommand{\cP}{{\mathcal P}}

\newcommand{\cR}{{\mathcal R}}

\newcommand{\cpH}{{^p\mathcal{H}}}

\newcommand{\rH}{{\rm H}}

\newcommand{\sE}{{\mathscr E}}

\newcommand{\sL}{{\mathscr L}}

\newcommand{\sO}{{\mathscr O}}

\newcommand{\Spec}{{\rm Spec}}

\DeclareMathOperator{\KK}{K}

\DeclareMathOperator{\im}{Image}

\DeclareMathOperator{\colim}{\mathrm{colim}}

\title{Perverse Filtrations via Brylinski-Radon transformations}
\author{Ankit Rai}
\address{Chennai Mathematical Institute, Chennai 603103, India.}
\email{ankitr@cmi.ac.in}
\author{K.V. Shuddhodan}
\address{Institut des Hautes \'Etudes Scientifiques, Universit\'e Paris-Scalay, CNRS, Laboratoire Alexandre Grothendieck, Le Bois-Marie 35 rte de Chartres, 91440 Bures-sur-Yvette, France}
\email{kvshud@ihes.fr}

\thanks{KVS was supported by the CARMIN project fellowship.}

\begin{document}
\maketitle
\begin{abstract}
In this article, we prove the $t$-exactness of a Brylinski-Radon transformation taking values in sheaves on flag varieties. This implies several weak Lefschetz type results for cohomology. In particular, we obtain de Cataldo-Migliorini's P=Dec(F) and Beilinson's basic lemma, the latter was an important ingredient in their proof of P=Dec(F). Our methods also allow the sharpening of Esnault-Katz's cohomological divisibility theorem and estimates for the Hodge level. Finally, we upgrade P=Dec(F) to an equivalence of functors which is also valid over a base.
\end{abstract}

\numberwithin{equation}{section}
\setcounter{tocdepth}{1}
\tableofcontents

\section{Introduction}

Let $Z$ be a variety\footnote{A variety is an irreducible and separated scheme of finite type over $k$.} of dimension $n$ and $\KK$ be a sheaf on it. Let $Z_{\bullet}:=\lbrace Z = Z_0 \supset Z_{-1} \supset Z_{-2} \supset \dots \supset Z_{-n} \supset Z_{-n-1} = \emptyset \rbrace$ be a filtration of $Z$ by closed subvarieties $Z_{-p}$ of codimension $p$ in $Z$. Using $Z_{\bullet}$ we may define a filtration, $F^{\bullet}$ on $\mathrm{H}^i(Z, \KK)$ called the \textit{flag filtration} (associated to $Z_{\bullet}$) as follows
\begin{equation} \label{eqn:geometric-filtration-def}
    F^j\mathrm{H}^i(Z, \KK) \colonequals \ker\left( \mathrm{H}^i(Z, \KK) \rightarrow \mathrm{H}^i(Z_{j-1}, \KK|_{Z_{j-1}})\right).
\end{equation}

On the other hand using the perverse truncation functors one can associate a natural decreasing filtration $P^{\bullet}$ on $\mathrm{H}^i(Z, \KK)$ called the \textit{perverse filtration} as follows
\begin{equation} \label{eqn:perverse-filtration-def}
    P^j\left( \mathrm{H}^i(Z, \KK)\right) \colonequals \im\left( \mathrm{H}^i(Z, {}^p\tau_{\leq -j}\KK) \rightarrow \mathrm{H}^i(Z, \KK)\right).
\end{equation}
One of the principal results of \cite{deCM10} is the following result comparing the above two filtrations.

\begin{thm}(de Cataldo-Migliorini)\cite[Theorem 4.1.1]{deCM10}
\label{thm:affine-perverse=geometric}
Let $Z \subseteq \bbA^N$ be an affine variety and let $\KK$ be a sheaf on $Z$. Let $Z_{\bullet}$ be a full flag as above, obtained by (repeatedly) intersecting $Z$ with generic hyperplane sections. Then 
\[
    P^{j}\mathrm{H}^i(Z, \KK) = F^{i+j}\mathrm{H}^i(Z, \KK). 
\]
\end{thm}

The proof of Theorem \ref{thm:affine-perverse=geometric} proceeds by constructing an isomorphism between the perverse spectral sequence and the shifted flag spectral sequence given by a generic flag, and obtaining the isomorphism on the filtrations as a corollary. A key vanishing result that allows them to compare these spectral sequences is the strong weak Lefschetz theorem\footnote{This is the name used in \cite[Theorem 5.1.2]{deCM10}. A form of the statement in \cite{Nor02} is called the basic lemma.} due to Beilinson \cite[Theorem 5.1.2]{deCM10}. They also prove a quasi-projective version of the above result \cite[Theorem 4.2.1]{deCM10}, using a similar strategy. The vanishing result needed again comes from the strong weak Lefschetz theorem. 

\subsection{Brylinski's Radon transform}The goal of this article is to study perverse filtrations using Brylinski's generalization of the Radon transforms \cite{Br}. These are sheaf theoretic analogues of the usual Radon transforms and take as input sheaves on $\bbP$ and produce a sheaf on the Grassmannian $\bG(d)$. In his article, Brylinski gave several applications of the transform to the Lefschetz theory \cite[I.III]{Br} and especially to the microlocal study of sheaves. More recently in his seminal article \cite{B}, Beilinson used Brylinski's Radon transform to give a construction of the singular support in the algebraic setting.

Our starting point is the observation that, there are strong restrictions on the perverse cohomologies of the Radon transforms of perverse sheaves on $\bbP$, that come from an affine scheme mapping quasi-finitely to $\bbP$. These restrictions suitably reinterpreted give rise to several vanishing statements. 

\subsubsection{Perverse filtrations in practice}An interesting special case of the perverse spectral sequence (and hence the perverse filtration) is when $\KK$ is a pushforward of a sheaf along a morphism, and in this case, the spectral sequence is the perverse analogue of the usual Leray spectral sequence. In recent years the perverse filtration has been an active area of research and is one-half of the eponymous P=W conjecture of de Cataldo–Hausel–Migliorini \cite{dHM12}. There is a rich body of work around this topic and we mention a few references here \cite{MS22}, \cite{HMMS22}, \cite{dMS22}. Finally, we conclude this section by fixing some notations and conventions

\subsubsection{\'etale sheaves}
Through this article, we work over an algebraically closed field $k$ and a prime $\ell$ invertible in $k$. We also fix a coefficient ring $\Lambda$ which is either a finite ring with torsion invertible in $k$ or an algebraic extension $E/\bbQ_{\ell}$. For any scheme $X/k$\footnote{For us schemes are always separated and finite type over the base field $k$.}, we denote by $D^b_c(X)$ the bounded derived category of constructible \'etale sheaves with finite tor-dimension and with coefficient in $\Lambda$. In what follows by a sheaf we simply mean an object in $D^b_c(X)$. We shall call a sheaf lisse if all its cohomology sheaves are lisse. Finally $D^b_c(X)$ comes equipped with the standard \cite[1.1.2, (e)]{Del80} and a perverse $t$-structure \cite[\S 4.0]{BBDG18}\footnote{Through this article we will deal exclusively with the perverse $t$-structure (with respect to the middle perversity).}. We shall denote the associated perverse truncation and cohomology functors by $(^p\tau^{\leq 0},{}^p\tau^{\geq 0})$ and $\cpH^i$ respectively.

\section{Summary of results}

Let $\bbP$ be the projective space of dimension $N$ over $k$. For an integer $d$ with $0 \leq d \leq N$ we denote the Grassmannian of $d$-planes in $\bbP$ by $\bG(d)$. Thus in this notation $\bG(N-1)$ is the dual space of $\bbP$, also denoted by $\bbP^{\vee}$. As a convention we set $\bG(d)=\emptyset$ for $d<0$ and $\bG(d)=\bG(N)$ for $d > N$. For $d \geq 0$ (resp. $d<0$) and a closed point $v \in \bG(d)$, we denote by $\Lambda_v$ the corresponding linear subspace (resp. the empty set).

\subsection{$\textup{P}=\textup{Dec(F)}$ on cohomology}

let $Z$ be a scheme and $\phi \colon Z \to \bbP$ be a quasi-finite morphism. Then the following result is proved in \S \ref{P_Dec(F)_Coho}.

\begin{thm}\label{thm:affine-radon1}
For any sheaf $\KK$ on $Z$, there exists an open dense $V \subseteq \bG(d)$ such that, for any closed point $v$ of $V$ (with $\Lambda_v$ the corresponding $d$-plane) and any integer $i$,
\begin{equation*}\label{affine-radon2}
P^{-i}\mathrm{H}^{d-N+1+i}(Z, \KK) \supseteq \ker \left(\mathrm{H}^{d-N+1+i}(Z, \KK) \to \mathrm{H}^{d-N+1+i}(\phi^{-1}(\Lambda_{v}),\KK|_{\phi^{-1}(\Lambda_{v})})\right).
\end{equation*}
Moreover, when $Z$ is affine, we may choose $V$ such that the above inclusion becomes an equality.
\end{thm}

We prove Theorem \ref{thm:affine-radon1} without resorting to spectral sequences, and in fact do so by fixing the codimension $d$, as opposed to dealing with full flag. By varying $d$ and renumbering the indices it is clear that Theorem \ref{thm:affine-radon1} implies Theorem \ref{thm:affine-perverse=geometric}. 

Our proof of Theorem \ref{thm:affine-radon1} follows from certain $t$-exactness properties of Brylinski-Radon transformations (see Proposition \ref{Grass_Rad_Corollary}). These as opposed to Theorem \ref{thm:affine-radon1} remain valid over a base. We now discuss a special case of Proposition \ref{Grass_Rad_Corollary}.

\subsection{$t$-exactness of $\cR_{N-1!} \circ \phi_{*}$ implies the basic lemma}
Consider the diagram

\[
\begin{tikzcd}
    & Q \arrow[r,hook,"i"] & \bbP \times \bbP^{\vee} \arrow[rd,"\pi^{\vee}"] \arrow[ld,"\pi"'] & U \arrow[l,hook',"j"'] \\
    Z \ar[r,"\phi"] & \bbP & & \bbP^{\vee},
\end{tikzcd}
\]
\noindent here $Q$ is the universal hyperplane section and $U$ is its open complement. Define a functor $\cR_{N-1!} \colon D^b_c(\bbP) \to D^b_c(\bbP^{\vee})$, as follows

\[
\cR_{N-1!}:=\pi^{\vee}_*j_!j^!\pi^{\dagger}\footnote{Give a smooth morphism $f:X \to Y$ of relative dimension $d$ with geometrically connected fibers, we set $f^{\dagger}:=f^*[d]$.}.
\]

A special case of Proposition \ref{Grass_Rad_Corollary} is the following result.

\begin{prop}\label{Cor_R_!_Intro}
    Let $\phi \colon Z \to \bbP$ be a quasi-finite morphism. Then the functor $\cR_{N-1!} \circ \phi_*$ is left $t$-exact. Moreover, when $Z$ is affine, it is $t$-exact.
\end{prop}

The proof of proposition \ref{Cor_R_!_Intro} is via a `double' Artin vanishing Lemma \ref{criterion_t_exactness_2}. We also refer the reader to Proposition \ref{SWL_converse} where it is shown that under the assumption of $Z$ affine, the functor $\cR_{N-1!} \circ \phi_*$  from $\cP(Z)$\footnote{For any scheme $X$, by $\cP(X)$ we mean the abelian category of perverse sheaves on $X$.} to $\cP(\bbP^{\vee})$ is in addition \textit{faithful}.

A consequence of Proposition \ref{Cor_R_!_Intro} is the strong weak Lefschetz Theorem from \cite{deCM10} due to Beilinson. 

\begin{thm}(Beilinson)\label{SWL_Intro}
Let $\phi: Z \to \bbP$ be a quasi-finite\footnote{In \cite{deCM10} this is stated for $\phi$ an affine open immersion.} and affine morphism. Let $\KK$ perverse. Then 
 \begin{enumerate}[(b)]
     \item the cohomology groups $\rH^{i}(Z,j_{1!}j_{1}^{!}j_{2*}j_{2}^*\KK)$ vanish for $i \neq 0$, here $j_{r} \colon V_r \hookrightarrow Z,r=1,2$ are inverse images (under $\phi$) of generic hyperplane section complements.
     \item If $Z$ is itself affine then, $\rH^i(Z,j_{!}j^!\KK)=0$ for $i \neq 0$. Here $j \colon V \hookrightarrow Z$ is the inverse image (under $\phi$) of a generic hyperplane section complement.
 \end{enumerate}
   
\end{thm}

Theorem \ref{SWL_Intro}, (b) follows from Proposition \ref{Cor_R_!_Intro}. While Theorem \ref{SWL_Intro}, (a) is obtained by using a relative version of Proposition \ref{Cor_R_!_Intro} and its dual. See \S \ref{proof_SWL} for details.

\subsection{Applications to cohomological divisibility and Hodge level}Proposition \ref{Cor_R_!_Intro} also implies Deligne's weak Lefschetz theorem \cite[Corollary A.5]{Kat93} under the weaker hypothesis of upper semi-perversity (see Corollary \ref{Deligne_WL}).
This allowed us to obtain a weak Lefschetz for the Gysin map (see Proposition \ref{affine_Gysin}), thus answering a question of Esnault-Wan \cite{EW22}. Then following a strategy by Esnault-Wan \cite{EW22}, we prove the following strengthening of Deligne's integrality theorem and Esnault-Katz's Cohomological divisibility theorem \cite[Theorem 2.3, (2)]{EK05}. 

\subsubsection{Statement on cohomological divisibility}Let $\bbF_{q}$ be a finite field with $q$ elements. Let $Z_0 \subset \bbA^N_{\bbF_{q}}$ be an affine scheme defined by the vanishing of $r$ polynomials of degree $d_1,d_2 \cdots, d_r$ with $N,r \geq 1$. We denote by $Z$ (resp. $\bbA^N$), the base change of $Z_0$ (resp. $\bbA^N_{\bbF_q}$) to an algebraic closure (say $\bar{\bbF}_q$) of $\bbF_q$. 

\begin{thm}\label{CD_Intro}
    The eigenvalues of the geometric Frobenius acting on 
\begin{enumerate}[(a)]
    \item $\rH_c^{\textup{dim}(Z)+j}(Z)$, for $j\geq 0$ and,
    \item $\rH_c^{\textup{dim}(Z)+j+1}(\bbA^N\backslash Z)$ for $0 \leq j \leq \textup{dim}(Z)+1$
\end{enumerate}
    \noindent are divisible by $q^{\mu_j(N;d_1,d_2 \cdots,d_r)}$ (see Equation (\ref{BCH_1})).
\end{thm}

We also have an analogue of Theorem \ref{CD_Intro} about Hodge levels. We refer the reader to Theorem \ref{HL} for the precise statement and to \S \ref{CD_Proof} for the proof of Theorem \ref{CD_Intro}.

\subsection{$\textup{P}=\textup{Dec(F)}$ as functors}
In the final \S \ref{P_Dec_F_Functors} we upgrade Theorem \ref{thm:affine-radon1} into an equivalence of functors. The results in the section are motivated by those in \cite[\S 5]{BMS19}. 

We make use of both geometric input (via Proposition \ref{flag_t_exactness}) and results from stable $\infty$-categories via Theorem \ref{Beilinson_t_structure} and Corollary \ref{t_strtucture_DF_cons}. The interested reader may refer to \S \ref{first_attempt} for a possible explanation as to why we needed to make a detour out of $1$-categories. 

We briefly describe the setup here. Let $S/k$ be a base scheme and let $\sE$ be a locally free sheaf of rank $N+1$ on $S$. We denote by $\bbP$ and $\textup{Fl}$ the associated projective bundle and flag bundle. Let $\overline{\pi}$ (resp.~$\overline{\pi}^{\vee}$) be the projection maps from $\bbP \times_S \textup{Fl}$ to $\bbP$ (resp. $\textup{Fl}$).

\subsubsection{The functor $\textup{F}$}
\begin{enumerate}[(a)]
    \item On $\bbP \times_S \textup{Fl}$ we have the \textit{universal} flag $\overline{Q}_{0}:=\bbP \times_S \textup{Fl} \supset \overline{Q}_{-1} \supset \overline{Q}_{-2} \cdots \overline{Q}_{-N} \supset \overline{Q}_{-N-1}=\emptyset$. 
    \item  For $-1 \geq r \geq -N-1$, we let $\overline{U}_{r}$ the complement of $\overline{Q}_{r}$, and denote by $\bar{l}_{r}:\overline{U}_{r} \hookrightarrow \bbP \times_S \textup{Fl}$, the corresponding open immersion. Thus we have adjunction maps $\overline{l}_{r!}\overline{l}_{r}^{!} \to \overline{l}_{r-1!}\overline{l}^{!}_{r-1}$.
\end{enumerate}

This allows us to define a functor $\textup{F} \colon \cD_{\textup{cons}}(\bbP) \to \cD F_{\textup{cons}}(\textup{Fl})$ as follows

\[
\KK \mapsto \textup{F}(\KK)^{\bullet}:=\{\cdots \to \textup{F}^{i}\KK \to \textup{F}^{i-1}\KK \to \textup{F}^{i-2}\KK \cdots\},
\]
\noindent where 

\begin{center}
   \begin{equation}\label{filt_defn_intro}
       \textup{F}^{r}\KK=\left \{
	\begin{array}{ll}
		\overline{\pi}^{\vee}_{*}\overline{\pi}^{\dagger}\KK   & \mbox{if } r \leq -N-1=-\textup{rk}(\sE) \\
		\overline{\pi}^{\vee}_{*}\overline{l}_{r-1!}\overline{l}_{r-1}^!\overline{\pi}^{\dagger}\KK & \mbox{if } -N \leq r \leq 0 \\
        0 & \mbox{if } r>0
	\end{array}.
\right.
\end{equation}
\end{center}
Here $\cD_{\textup{cons}}(\bbP)$ and $\cD F_{\textup{cons}}(\textup{Fl})$ are symmetric monoidal stable $\infty$-categories. The homotopy category of $\cD_{\textup{cons}}(\bbP)$ is $D^b_c(\bbP)$, and the objects in $\cD F_{\textup{cons}}(\textup{Fl})$ can be thought of as constructible complexes with a finite filtration. For a precise definition see \S \ref{cDF_cons}.

\subsubsection{The functor $\textup{P}$}
We also define a functor $\textup{P} \colon \cD_{\textup{cons}}(\bbP) \to \cD F_{\textup{cons}}(\textup{Fl})$ as follows

\[
\KK \mapsto \textup{P}(\KK)^{\bullet}:=\{\cdots \to \textup{P}^{i}\KK \to \textup{P}^{i-1}\KK \to \textup{P}^{i-2}\KK \cdots\},
\]
\noindent where 
\[
\textup{P}^{i}\KK:=\overline{\pi}^{\vee}_{*}\overline{\pi}^{\dagger}{}^p\tau_{\leq -i}\KK.
\]

\subsubsection{Beilinson $t$-structure on $\cD F_{\textup{cons}}(\textup{Fl})$ and the functor $\textup{Dec}$}
In the appendix, we equip $\cD F_{\textup{cons}}(\textup{Fl})$ with the following additional structures.
\begin{enumerate}[(a)]
    \item A Beilinson $t$-structure (see Theorem \ref{Beilinson_t_structure} and Proposition \ref{t_strtucture_DF_cons}).
    \item A functor $\textup{Dec}$ (see Definition \ref{decalage} and Proposition \ref{t_strtucture_DF_cons}), which is to mimic Deligne's  décalée \cite[1.3.3]{Del71}.
\end{enumerate}

\subsubsection{Statement of the Theorem}We can now state our result. For a proof see \S \ref{P_Dec_F_Functors}.
\begin{thm}\label{P_Dec_F_Functor_intro}
Let $\phi \colon Z \to \bbP$ be a quasi-finite morphism with $Z/S$ affine. Then 
\begin{enumerate}[(a)]
    \item $\textup{F}\circ \phi_*$ is $t$-exact, for the perverse $t$-structure on $\cD_{\textup{cons}}(\bbP)$ and Beilinson $t$-structure on $\cD F_{\textup{cons}}(\textup{Fl})$.
    \item There is a natural equivalence of functors
  \[
  \mathrm{P} \circ \phi_* \simeq \mathrm{Dec}(\mathrm{F}\circ\phi_*).
  \]
  \item There is a functorial (in $\KK$) isomorphism of spectral sequences with values in $\cP(\textup{Fl})$ between the constant perverse spectral sequence 

  \[
   E_1^{i,j}=\cpH^{i+j}(\overline{\pi}^{\vee}_{*}\overline{\pi}^{\dagger}\cpH^{-j}(\phi_*\KK)[j]) \implies \cpH^{i+j}(\overline{\pi}^{\vee}_{*}\overline{\pi}^{\dagger}\phi_*\KK)
   \]
   \noindent and the  d\'ecal\'ee of the universal flag filtration spectral sequence
   \[
   E_1^{i,j}=\cpH^{i+j}(\textup{Gr}^{-j}_{\textup{F}}\phi_*\KK[j]) \implies \cpH^{i+j}(\overline{\pi}^{\vee}_{*}\overline{\pi}^{\dagger}\phi_*\KK).
   \]
   
\end{enumerate}
  
\end{thm}

When $S=\Spec(k)$, by localizing at a generic point of $\textup{Fl}$ we recover the result in \cite[Theorem 4.1.1]{deCM10}.

{\bf Acknowledgements:}

This work owes an overwhelming intellectual debt to Beilinson's article \cite{Bei87}. Though we do not use the results from \cite{B} in this article, its masterful use of the Radon transform was very insightful for us. KVS would like to thank Charanya Ravi and Donu Arapura for their helpful discussions. 

\newpage 
\subsection{Leitfaden}

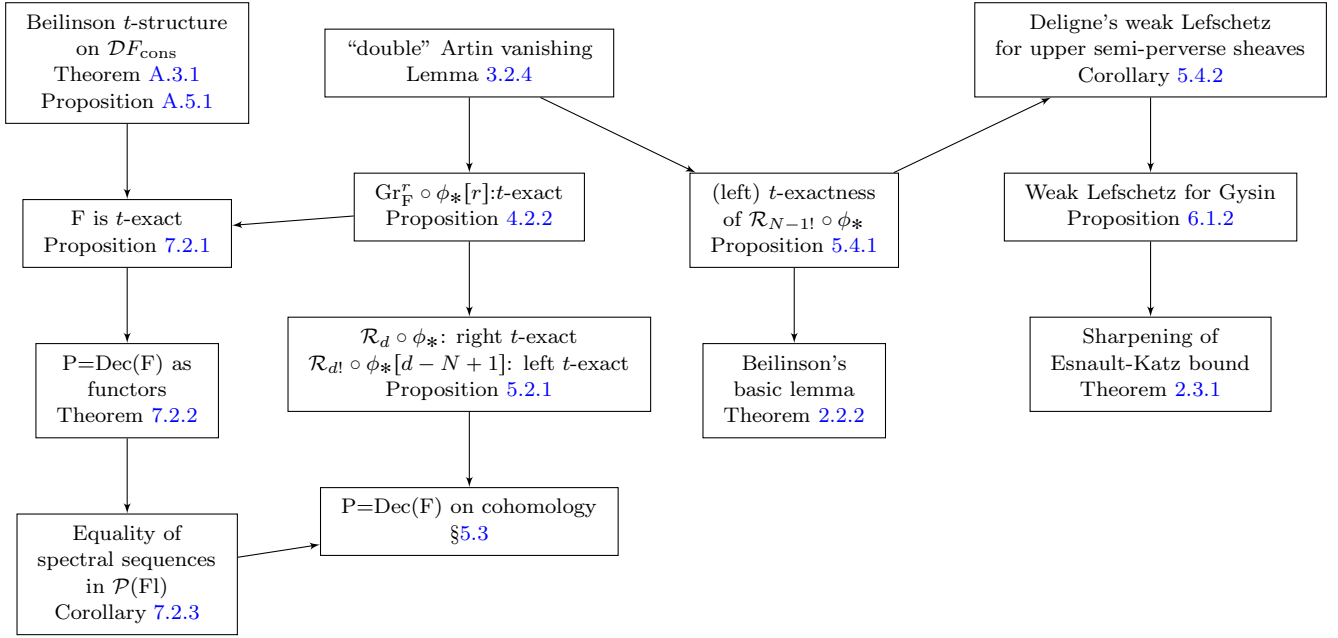
\begin{center}
\tiny \begin{tikzpicture}[
        > = latex',
        %
        every node/.style={
            draw,
            minimum width=1cm,
            minimum height=0.8cm,
        },
    ]
        \node (S1) {\begin{tabular}{c}
        ``double'' Artin vanishing \\
     Lemma \ref{criterion_t_exactness_2} \\
    \end{tabular} };
        \node[below=of S1] (S2) {\begin{tabular}{c}
        $\textup{Gr}^{r}_{\textup{F}}\circ\phi_*[r]$:$t$-exact\\
     Proposition \ref{flag_t_exactness}  \\
    \end{tabular} };
        \node[below=of S2] (S3) {
        \begin{tabular}{c}
        $\cR_{d} \circ \phi_{*}$: right $t$-exact \\
        $\cR_{d!} \circ \phi_{*}[d-N+1]$: left $t$-exact\\
     Proposition \ref{Grass_Rad_Corollary} \\
    \end{tabular} 
             };
        \node[below=of S3] (S4) {\begin{tabular}{c}
             P=\textup{Dec}(F) on cohomology\\
             \S \ref{P_Dec(F)_Coho}
        \end{tabular}};
        \node[below right=of S1] (S5) {\begin{tabular}{c}
             (left) $t$-exactness \\
             of $\cR_{N-1!}\circ \phi_*$\\
             Proposition \ref{SWL_Rad}
        \end{tabular}};
        \node[below=of S5] (S11) {\begin{tabular}{c}
             Beilinson's \\
             basic lemma\\
             Theorem \ref{SWL_Intro}
        \end{tabular}};
        \node[above right=of S5] (S6) {\begin{tabular}{c}
             Deligne's weak Lefschetz\\
             for upper semi-perverse sheaves\\
             Corollary \ref{Deligne_WL}
        \end{tabular}};
        \node[below=of S6] (S7) {\begin{tabular}{c}
       Weak Lefschetz for Gysin \\
        Proposition \ref{affine_Gysin2}
        \end{tabular}};
        \node[below=of S7] (S8) {\begin{tabular}{c}
             Sharpening of\\
             Esnault-Katz bound\\
             Theorem \ref{CD_Intro}
        \end{tabular}};
        
        \node[left=of S1] (S12) {\begin{tabular}{c}
            Beilinson $t$-structure\\
            on $\cD F_{\textup{cons}}$\\
            Theorem \ref{Beilinson_t_structure} \\
            Proposition \ref{t_strtucture_DF_cons}
        \end{tabular}};
        \node[below=of S12] (S14) {\begin{tabular}{c}
            $\textup{F}$ is 
            $t$-exact\\
           Proposition \ref{prop_revist}
        \end{tabular}};
         \node[below=of S14] (S10) {\begin{tabular}{c}
            P=\textup{Dec}(F) as\\
            functors\\
            Theorem \ref{P_Dec_F_Functor}
        \end{tabular}};
        \node[below=of S10] (S13) {\begin{tabular}{c}
            Equality of \\
            spectral sequences\\
            in $\cP(\textup{Fl})$\\
            Corollary \ref{iso_spectral_sequence}
        \end{tabular}};

        \draw[->] 
              (S1) edge (S2)
              (S14) edge (S10)
              (S12) edge (S14)
              (S2) edge (S14)
              (S10) edge (S13)
              (S13) edge (S4)
              (S5) edge (S11)
              (S1) edge (S5)
              (S5) edge (S6)
              (S6) edge (S7)
              (S7) edge (S8)
              (S2) edge (S3)
              (S3) edge (S4);
    \end{tikzpicture}
\end{center}

\section{Preliminary results}

In this section, we collect the preliminary results we need to carry out the proof of the main theorems of this article. We begin with a summary of the results needed from the theory of perverse sheaves. 

\subsection{Summary of results needed about perverse sheaves}

As mentioned in the introduction we work exclusively with middle perversity. We shall make repeated use of Artin vanishing in the following form. 

\begin{thm}\cite[Theorem 4.1.1]{BBDG18}\label{Artin_Vanishing}
Let $f: X \to Y$ be an affine morphism. Then the functor $Rf_*: D^b_c(X) \to D^b_c(Y)$\footnote{We work with triangulated versions of the sheaf operations, and henceforth shall denote them without the usual decorations of ‘R or ‘L’.} (resp. $f_!: D^b_c(X) \to D^b_c(Y)$)  is right $t$-exact (resp. left $t$-exact).
\end{thm}

We shall also make repeated use of the $t$-exactness of the following functors without further mention.

\begin{prop}\label{t-exact_collection}
The following are $t$-exact.
\begin{enumerate}[(a)]
\item The functor $f^{\dagger}$ for a smooth morphism $f$ of relative dimension $d$ with geometrically connected fibers.

\item The functors $f_!$ and $f_*$ for a quasi-finite affine morphism $f$.

\end{enumerate}

\end{prop}

\begin{proof}
(a) is \cite[Proposition 4.2.5]{BBDG18} and (b) is \cite[Proposition 2.2.5]{BBDG18} combined with Artin vanishing.
\end{proof}

The following lemma is standard and proved here for ease of exposition. 

\begin{lem}\label{perverse_usual}
Let $X/k$ be a smooth variety and $\KK$ be a lisse sheaf on $X$. Then there is natural isomorphism $\cpH^{i}(\KK)=\ccH^{i}(\KK[-\textup{dim}(X)])[\textup{dim}(X)]$.
\end{lem}

\begin{proof}
We induct on the number $N$ of non-zero $\ccH^i(\KK)$. The statement is true when $N \leq 1$. Now assume $N>1$ and let $r$ be the least integer such that $\ccH^r(\KK) \neq 0$. We have an exact triangle $\xymatrix{\ccH^{r}(\KK)[-r] \ar[r] & \KK \ar[r] & \tau_{>r}\KK \ar^{+1}[r]&}$.

 Note that for $i\neq r$ we have $\ccH^i(\KK) \simeq \ccH^{i}(\tau_{>r}\KK)$ and $\cpH^{i+\textup{dim}(X)}(\KK)=\cpH^{i+\textup{dim}(X)}(\tau_{>r}\KK)$. The former equality follows from the definition of $\tau_{>r}$ and the latter follows from perversity of $\ccH^{r}(\KK)[\textup{dim}(X)]$ combined with $\cpH^{i+\textup{dim}(X)}(\tau_{>r}\KK)=0$ for $i \geq r$ (by induction hypothesis).
\end{proof}

\subsection{A criteria for $t$-exactness}\label{vanishing_lemma}

 Now we shall prove a general criterion for $t$-exactness of certain functors which shall be crucially used in the proof of Proposition \ref{flag_t_exactness}. Consider the following diagram of schemes over $k$.

\begin{equation}\label{vanishing_diagram}
  \begin{tikzcd}[row sep=normal, column sep=normal]
      & U \arrow[r, hook, "j_{U}"] \arrow[d, "\phi_U"] & X \arrow[d, "\phi_X"] & Y \arrow[l, hook', "i_{Y}"'] \arrow[d,"\phi_Y"]\\
      & \overline{U} \arrow[r, hook, "j_{\overline{U}}"] & \overline{X} \arrow[dr,"\pi_{\overline{X}}^{\vee}"] \arrow[dl,"\pi_{\overline{X}}"'] & \overline{Y} \arrow[l, hook', "i_{\overline{Y}}"'] \\
      Z \arrow[r,"\phi_Z"] & \bar{Z} & & \bar{Z}^{\vee} 
  \end{tikzcd}
\end{equation}

We assume the following hold.

\begin{enumerate}[(a)]
     \item $j_{\overline{U}}$ is an open immersion and $i_{\overline{Y}}$ the corresponding closed complement.
     \item $\pi_{\overline{X}}^{\vee}$ is proper. The morphisms $\pi_{\overline{X}}$ and $\pi_{\overline{Y}}:=\pi_{\overline{X}} \circ i_{\overline{Y}}$ are smooth with geometrically connected fibers.
     \item $U \hookrightarrow X \hookleftarrow Y$ is obtained from $\bar{U} \hookrightarrow \bar{X} \hookleftarrow \bar{Y}$ by base changing along $\phi_Z$.
\end{enumerate}

\begin{lem}\label{vanishing_lemma_sta}
    For a sheaf $L$ on $X$, the natural map $j_{\overline{U}!}\phi_{U*}j_{U}^{!}L \to \phi_{X*}j_{U!}j_{U}^{!}L$ is an isomorphism iff the natural map $i_{\overline{Y}}^*\phi_{X*}L \to \phi_{Y*}i^{*}_{Y}L$ is an isomorphism.
\end{lem}

\begin{proof}
    The cone of the natural map $\alpha \colon j_{\overline{U}!}\phi_{U*}j_{U}^{!}L \to \phi_{X*}j_{U!}j_{U}^{!}L$ is $i_{\overline{Y}*}i_{\overline{Y}}^*\phi_{X*}j_{U!}j_{U}^{!}L$. Hence $\alpha$ is an isomorphism iff $i_{\overline{Y}}^*\phi_{X*}j_{U!}j_{U}^{!}L$ vanishes. The vanishing holds iff $i_{\overline{Y}}^*\phi_{X*}L \to i_{\overline{Y}}^*\phi_{X*}i_{Y*}i^{*}_{Y}L=\phi_{Y*}i^{*}_{Y}L$ is an isomorphism. 
\end{proof}

We continue working with the diagram (\ref{vanishing_diagram}). Let $\pi_X :X \to Z$ (resp. $\pi_Y: Y \to Z$) be the map induced by base change.

\begin{lem}\label{vanishing_lem_sta_2}
    For any sheaf $K$ on $Z$, the sheaf $L:=\pi_X^{*}K$ satisfies $i_{\overline{Y}}^*\phi_{X*}L \simeq \phi_{Y*}i^{*}_{Y}L$ under the natural map.
\end{lem}

\begin{proof}
Using (b), by smooth base change $\pi_{\overline{X}}^{*}\phi_{Z*}K \simeq \phi_{X*}\pi^{*}K$, under the natural map. Hence $\pi_{\overline{Y}}^*\phi_{Z*}K=i_{\overline{Y}}^{*}\pi_{\overline{X}}^{*}\phi_{Z*}K \simeq i_{\overline{Y}}^*\phi_{X*}\pi^{*}K$ under the natural map. Since $\pi_{\overline{Y}}$ is also smooth by assumption, the result follows from another application of smooth base change.
\end{proof}

Combining Lemmas \ref{vanishing_lemma_sta} and \ref{vanishing_lem_sta_2} we obtain the following corollary. 

\begin{cor}\label{R_!_description}
The functor $\cR_{!} \colonequals \pi_{\overline{X}*}^{\vee}j_{\overline{U}!}j_{\overline{U}}^!\pi_{\overline{X}}^{\dagger}\phi_{Z*}$ is naturally equivalent to $(\pi_{\overline{X}}^{\vee}\circ \phi_X)_* j_{U!}j_{U}^{!}\pi_X^{\dagger}$.

\end{cor}    

Consider the following additional conditions on the Diagram (\ref{vanishing_diagram}).
\begin{enumerate}
      \item[(d)] The morphism  $\phi_Z: Z \to \bar{Z}$ is a quasi-finite morphism, and  $\pi_{\overline{X}}^{\vee} \circ j_{\overline{U}}$ is affine.
     \item[(e)] The morphism $\pi_{\overline{X}}^{\vee} \circ \phi_X$ is affine.
\end{enumerate}
Now we can state the main lemma of this section. This abstracts out the `double' Artin vanishing needed for the $t$-exactness of the Radon transforms with source an affine scheme.

\begin{lem}\label{criterion_t_exactness_2}
    With the notations as above if (a)-(d) are satisfied then $\cR_{!}$ is left $t$-exact. Moreover if (e) also holds then $R_!$ is $t$-exact.
\end{lem}

\begin{proof}
By (b) above, since $\pi_{\overline{X}}^{\vee}$ is proper we can rewrite $\cR_{!}=(\pi_{\overline{X}}^{\vee} \circ j_{\overline{U}})_!j_{\overline{U}}^!\pi_{\overline{X}}^{\dagger}\phi_{S*}$. Since $\pi_{\overline{X}}^{\vee} \circ j_{\overline{U}}$ is affine by assumption ((c) above), Artin vanishing implies $(\pi_{\overline{X}}^{\vee} \circ j_{\overline{U}})_!$ is left $t$-exact. Since $\phi_S$ is quasi-finite, \cite[Proposition 2.2.5]{BBDG18} implies that $\phi_{S*}$ is left $t$-exact. Thus $\cR_{!}$ is left $t$-exact.

When (e) holds, using Corollary \ref{R_!_description} it suffices to show that $(\pi_{\overline{X}}^{\vee}\circ \phi_X)_* j_{U!}j_{U}^{!}\pi_X^{\dagger}$ is right $t$-exact. By \cite[\S 4.2.4]{BBDG18}, we know that $j_{U_!}$ is right $t$-exact. Moreover since $\pi_{\overline{X}}^{\vee}\circ \phi_X$ is affine ((e) above), Artin vanishing implies the right $t$-exactness of $(\pi_{\overline{X}}^{\vee}\circ \phi_X)_*$. Thus $\cR_{!}$ is right $t$-exact. This proves the exactness of $\cR_{!}$.
\end{proof}

\section{Graded quotients of a Brylinski-Radon transformation}\label{Brylinski-Radon_Flag}

We begin by setting up notations and recalling some basic facts about flag bundles \cite[Expose VI, \S 4]{SGA6}. In the rest of the article, let $S$ be the base scheme. 

\subsection{Recollection about Flag bundles}\label{Flag_Variety_Notation}

\begin{enumerate}[(a)]

    \item Let $\sE$ be a locally free sheaf of rank $N+1$ on $S$. Let $\bbP(\sE)$ be the associated projective bundle. Thus for any scheme $f \colon X \to S$, the set $\bbP(\sE)(X)$ is the collection of rank $1$ locally free quotients of $f^*\sE$ on $X$.

    \item We denote by $\textup{Fl}(\sE)$\footnote{We shall suppress $\sE$ from the notation when there is no scope of confusion.} the full flag bundle of $\sE$. More precisely $\textup{Fl}$ represents the functor defined on schemes over $S$, which to a scheme $f: X \to S$ associates the set $(\sL_{r})_{N \geq r \geq 1}$ of locally free sheaves on $X$ such that for all $r$, $\sL_r$ is of rank $r$ and is a quotient of $\sL_{r+1}$. Here we set $\sL_{N+1}=f^*\sE$. Note that $\textup{Fl}$ is smooth over $S$.

    \item We have the \textit{universal} flag $\overline{Q}_{0}:=\bbP \times_S \textup{Fl} \supset \overline{Q}_{-1} \supset \overline{Q}_{-2} \cdots \overline{Q}_{-N} \supset \overline{Q}_{-N-1}=\emptyset$. For any scheme $f \colon X \to S$, $\overline{Q}_{r}(X) \subseteq \bbP(X) \times \textup{Fl}(X)$ consisting of those one dimensional quotients of $f^{*}\sE$ which are also quotients of $\sL_{N+r}$ in the above notation. 

    \item The closed immersion $\overline{Q}_{r} \hookrightarrow \bbP \times_S \textup{Fl}$ is regular of  codimension $-r$, and moreover $\overline{Q_{r}}$ is smooth over $S$. The maps $\overline{Q}_{r} \backslash \overline{Q}_{r-1} \to \textup{Fl}(N)$ are affine (with fibers isomorphic to $\bbA^{N+r}$). These statements can be verified by Zariski locally over $S$, in particular, when $\sE=\sO_S^{\oplus N+1}$, where they are obvious.

     \item  For $-1 \geq r \geq -N-1$, we let $\overline{U}_{r}$ the complement of $\overline{Q}_{r}$, and denote by $\bar{l}_{r}:\overline{U}_{r} \hookrightarrow \bbP \times_S \textup{Fl}$, the corresponding open immersion. Thus we have adjunction maps $\overline{l}_{r!}\overline{l}_{r}^{!} \to \overline{l}_{r-1!}\overline{l}^{!}_{r-1}$.
     
     \item  For $-1 \geq r \geq -N-1$, we denote by $\overline{l}_{r,r-1} \colon \overline{Q}_{r} \backslash \overline{Q}_{r-1} \hookrightarrow \overline{Q}_{0}$ the locally closed immersion.

     \item Finally let  $\overline{\pi}$ (resp. $\overline{\pi}^{\vee}$) the projection map from $\bbP \times_S \textup{Fl}$ to $\bbP$ (resp. $\textup{Fl}$).

\end{enumerate}

\subsection{$t$-exactness of the graded pieces of a Brylinski-Radon transformation}

We continue using the notations from \S \ref{Flag_Variety_Notation}. Consider the functors $\textup{F}_{S}^{r}$ and $\textup{Gr}^{r}_{\textup{F}_S}$\footnote{We shall suppress $S$ from the notation when there is no scope for confusion.}  from $D^{b}_c(\bbP)$  to $D^b_c(\textup{Fl})$ defined below

   \begin{center}
   \begin{equation}\label{filt_defn}
       \textup{F}^{r}(\KK)=\left \{
	\begin{array}{ll}
		\overline{\pi}^{\vee}_{*}\overline{\pi}^{\dagger}\KK   & \mbox{if } r \leq -N-1=-\textup{rk}(\sE) \\
		\overline{\pi}^{\vee}_{*}\overline{l}_{r-1!}\overline{l}_{r-1}^!\overline{\pi}^{\dagger}\KK & \mbox{if } -N \leq r \leq 0 \\
        0 & \mbox{if } r>0 \\
	\end{array}
\right.
\end{equation}
\noindent and
\begin{equation}\label{grad_defn}
       \textup{Gr}^{r}_{\textup{F}}(\KK)=\left \{
	\begin{array}{ll}
		0  & \mbox{if } r \leq -N-1=-\textup{rk}(\sE) \\
		\overline{\pi}^{\vee}_{*}\overline{l}_{r,r-1!}\overline{l}_{r,r-1}^{*}\overline{\pi}^{\dagger}\KK & \mbox{if } -N \leq r \leq 0 \\
        0 & \mbox{if } r>0 \\
	\end{array}
\right..
\end{equation}

\end{center}

Now note that by definition one has a triangle

\begin{equation}\label{basic_Gr_triangle}
\small \xymatrix{\textup{F}^{r+1}\KK\ar[r] & \textup{F}^{r}\KK\ar[r] & \textup{Gr}^{r}_{\textup{F}}\KK \ar^-{+1}[r]&}.
\end{equation}

\begin{rem}\label{Nota_Sugge}
    Our notation here is suggestive and we would like to think of $\textup{F}$ as a functor to the filtered derived category of $\textup{Fl}$. However until \S \ref{P_Dec_F_Functors}, we will only need results about $\textup{Gr}^{r}_{\textup{F}}$, and hence avoid the enhancement until required.
\end{rem}

\begin{prop}\label{flag_t_exactness}
    Let $\phi \colon Z \to \bbP$ be a quasi-finite morphism. Then $\textup{Gr}^{r}_{\textup{F}}\circ \phi_{*}[r]$ is left $t$-exact. Moreover, if $Z$ is affine over $S$, then it is $t$-exact.
\end{prop}

\begin{proof}
    
    We denote by $Q_{r}$ (resp. $l_{r}$) the base change of $\overline{Q}_{r}$ (resp. $\overline{l}_{r}$) along $Z \times_S \textup{Fl} \to \bbP \times_S \textup{Fl}$. We have a diagram

    \begin{equation}\label{flag_t_exactness_diagram}
  \begin{tikzcd}[row sep=normal, column sep=normal]
      & Q_{r} \backslash Q_{r-1}  \arrow[r, hook, "j_{r}"] \arrow[d] & Q_{r} \arrow[d,"\phi_{Q_{r}}"] & Q_{r-1} \arrow[l, hook', "i_{r-1}"'] \arrow[d]\\
      & \overline{Q}_{r} \backslash \overline{Q}_{r-1} \arrow[r, hook, "\overline{j}_{r}"] & \overline{Q}_{r} \arrow[dr,"\overline{\pi}_{r}^{\vee}"] \arrow[dl,"\overline{\pi}_{r}"'] & \overline{Q}_{r-1} \arrow[l, hook', "\overline{i}_{r-1}"'] \\
      Z \arrow[r,"\phi"] & \bbP & & \textup{Fl},
  \end{tikzcd}
\end{equation}

which satisfies conditions (a)-(d) from \S \ref{vanishing_lemma}. Moreover, if $Z/S$ is affine, then so is $\overline{\pi}^{\vee}_{r} \circ \phi_{Q_{r}}$ and hence the diagram also satisfies condition (e) from \S \ref{vanishing_lemma}. Now note that 
\begin{equation}\label{flag_t_exactness_iso}
\textup{Gr}^{r}_{\textup{F}}\circ \phi_{*}[r]=\overline{\pi}^{\vee}_{*}\overline{l}_{r,r-1!}\overline{l}_{r,r-1}^{*}\overline{\pi}^{\dagger}\phi_{*}\KK[r]=\overline{\pi}_{r*}^{\vee}\overline{j}_{r!}\overline{j}_{r}^{!}\overline{\pi}_{r}^{\dagger}\phi_{*}\KK.
\end{equation}

The first equality holds by definition of $\textup{Gr}^r_F$, while the second follows from the fact that $\overline{\pi}_{r}^{\dagger}=\overline{\pi}^{\dagger}[r]$, owing to $Q_{r} \hookrightarrow Q_{0}$ being a regular immersion of codimension $-r$ (see \S \ref{Flag_Variety_Notation}, (d)). The last equality combined with Lemma \ref{criterion_t_exactness_2} implies the proposition.

\end{proof}

\section{Brylinski-Radon transformation with values in $D^b_c(\bG(d))$}\label{Brylinski-Radon}

In this section, we use Proposition \ref{flag_t_exactness} to derive consequences for the usual Brylinski-Radon transform (\cite[I, \S V]{Br},\cite[\S 2.2.2]{DS23}). These in particular will imply Theorem \ref{thm:affine-radon1}. We begin by recalling some basic properties of the Grassmannian. As before we fix a base scheme $S$ finite type over $k$ and a locally free sheaf $\sE$ of rank $N+1$ on $S$. Let $\bbP$ be the associated projective bundle. 

\subsection{Recollection about Grassmannians}\label{Grassmannians}
\begin{enumerate}[(a)]
    \item Let $\bG(d,\sE)$\footnote{As before we suppress $\sE$ from the notation when there is no scope for confusion.} be the Grassmannian scheme over $S$. More precisely, for any scheme $f \colon X \to S$, the set $\bG(d)(X)$ is the collection of rank $d+1$ locally free quotients of $f^*\sE$. $\bG(d)$ is smooth of relative dimension $(d+1)(N-d)$ over $S$. 

    \item Let $Q_{d} \subset \bbP \times_S \bG(d)$ be the incidence correspondence and $U_{d}$ its complement. The inclusion map from $Q_d$ (resp. $U_d$) into $\bbP \times_S \bG(d)$ is denoted by $i_d$ (resp. $j_d$).

    \item We denote by $\pi_{d}$ (resp. $\pi^{\vee}_{d}$) the maps from $\bbP \times_S \bG(d)$ to $\bbP$ (resp. $\bG(d)$). 
    
    \item The morphism $\pi_{d}$ is smooth and projective of relative dimension $d(N-d)$, and $\pi^{\vee}_{d}$ identifies $Q_{d}$ with a projective sub bundle of the trivial projective bundle $\bbP \times_S \bG(d)$ over $\bG(d)$. Thus $\pi^{\vee}_{d}$ is also smooth and projective of relative dimension $d$.

\end{enumerate}

\subsubsection{Definition of the Brylinski-Radon transformation}\label{def_BR}
Using the notations as above, we have the Brylinski-Radon transform \cite[I, \S V]{Br} from $D^b_c(\bbP)$ to $D^b_c(\bG(d))$,

\begin{equation}\label{Radon}
    \cR_{S,d}:=\pi^{\vee}_{d*}i_{d*}i^{*}_{d}\pi_{d}^{\dagger}[d-N].
\end{equation}

For this article it would be useful to study the following (modified) Brylinski-Radon transforms,

\begin{equation}\label{!Radon}
    \cR_{S,d!}:=\pi^{\vee}_{d*}j_{d!}j^{!}_{d}\pi_{d}^{\dagger},
\end{equation}
\noindent

\noindent and

\begin{equation}\label{*Radon}
     \cR_{S,d*}:=\pi^{\vee}_{d*}j_{d*}j^{*}_{d}\pi_{d}^{\dagger}.
\end{equation}

\subsubsection{Basic properties of the Brylinski-Radon transformation}\label{basic_prop_BR}
\begin{enumerate}[(a)]
    \item Note that (upto Tate twists)  $\cR_{d*}$\footnote{As before we suppress $S$ from the notation when there is no scope for confusion.} is dual to $\cR_{d!}$. Also $\cR_{d}$ is self-dual (up to Tate twists\footnote{Tate twists will only be relevant in \S 6, where we shall keep track of them.}).

    \item The triple $(Q_{d},\bbP \times_S \bG(d),U_d)$ gives rise to a triangle 
\begin{center}
    $\xymatrix{i_{d*}i_{d}^*[-1] \ar[r] & j_{d!}j_{d}^! \ar[r] & 1 \ar^{+1}[r]&}$,
\end{center}
\noindent
and hence to a triangle on $\bG(d)$
\begin{equation}\label{basic_triangle}
    \xymatrix{\cR_{d}\ar[r] & \cR_{d!}[d-N+1] \ar[r] & \pi^{\vee}_{d*}\pi^{\dagger}_{d}[d-N+1] \ar^-{+1}[r] &}.
\end{equation}
    
\end{enumerate}

\subsection{$t$-exactness properties of $\cR_{d?}\circ \phi_{*}$ }
 We use Proposition \ref{flag_t_exactness} to deduce Proposition \ref{Grass_Rad_Corollary} below. This is then used in \S \ref{P_Dec(F)_Coho} to give a proof of $\textup{P}=\textup{Dec(F)}$ at the level of cohomology. Later in \S \ref{d_N-1}, we will see further applications of Proposition \ref{Grass_Rad_Corollary}, in the case $d=N-1$ (i.e.~the usual Radon transform).

\begin{prop}\label{Grass_Rad_Corollary}
    Let $\phi \colon Z \to \bbP$ be a quasi-finite morphism. Then 
    \begin{enumerate}[(a)]
        \item  $\cR_{d!}[d-N+1] \circ \phi_{*}$ is left $t$-exact.
        \item  Moreover if $Z/S$ is affine, $\cR_{d!} \circ \phi_{*}$ and $\cR_{d} \circ \phi_{*}$ are right $t$-exact.
    \end{enumerate}
\end{prop}

\begin{proof}

 Let $f_d \colon \textup{Fl} \to \bG(d)$ be the projection map. This is a smooth and proper morphism and we have a commutative diagram (using notations from \S \ref{Flag_Variety_Notation}) 

\begin{equation}\label{Grass_Rad_Corollary_diagram}
\begin{tikzcd}[row sep=normal, column sep=normal]
   \overline{U}_{-N+d} \arrow[r,hook,"\overline{l}_{-N+d}"] \arrow[d] & \bbP \times_S \textup{Fl} \arrow[d,"1\times_S f_d"] \arrow[r] & \textup{Fl} \arrow[d,"f_d"] \\
    U_{d} \arrow[r,hook,"j_{d}"] & \bbP \times_S \bG(d) \arrow[r] & \bG(d).
\end{tikzcd}
\end{equation}  

Since $f_d$ is a cover in the smooth topology, thus by \cite[Proposition 4.2.5]{BBDG18} the $t$-exactness (left or right) of $\cR_{d?}\circ \phi_{*}$ (or its shifts) can be checked after applying $f_d^{\dagger}$. Now by smooth base change, we get (see Equation (\ref{filt_defn})) 
\begin{equation}\label{pullback_Gra_Rad}
f_{d}^\dagger \left ( \cR_{d!}\circ \phi_{*} \right )=\textup{F}^{-N+d+1}\circ \phi_{*}.
\end{equation}

\textit{Case: $\phi$ quasi-finite}\newline
Using the isomorphism (\ref{pullback_Gra_Rad}), to prove (a) it suffices to show that for any $0 \leq d \leq N-1$, $\textup{F}^{-N+d+1}\circ \phi_{*}[-N+d+1]$ is left $t$-exact. For every integer $r$ and sheaf $\KK$ on $Z$, we have an exact triangle

 \[
 \small \xymatrix{\textup{F}^{r}\circ \phi_{*}(\KK) [r] \ar[r] & \textup{Gr}_\textup{F}^{r}\circ \phi_{*}(\KK) [r] \ar[r] & \textup{F}^{r+1}\circ \phi_{*}(\KK)[r+1] \ar^-{+1}[r]&}.
\]

Thus $\textup{F}^{r}\circ \phi_{*}[r]$ is left $t$-exact if $\textup{F}^{r+1}\circ \phi_{*}[r+1]$ and $\textup{Gr}_\textup{F}^{r}\circ \phi_{*}[r]$ are so. By Proposition \ref{flag_t_exactness} and by descending induction (on $r$) we are reduced to showing the left $t$-exactness for the case $r=0$ (or $d=N-1$), in which case $\textup{F}^{0}\circ \phi_{*}=\textup{Gr}_\textup{F}^{0}\circ \phi_{*}$, and the left $t$-exactness follows from Proposition \ref{flag_t_exactness}.

\textit{Case: $\phi$ quasi-finite and $Z/S$ affine} \newline
Using the isomorphism (\ref{pullback_Gra_Rad}) as before, to show right $t$-exactness of $\cR_{d!} \circ \phi_{*}$ for any $d$, it suffices to show the right $t$-exactness of $\textup{F}^{r}\circ \phi_{*}$ for any $r$. For any sheaf $\KK$ on $Z$ we have an exact triangle

\[
 \small \xymatrix{\textup{F}^{r+1}\circ \phi_{*}(\KK)[r] \ar[r] & \textup{F}^{r}\circ \phi_{*}(\KK)[r] \ar[r] & \textup{Gr}_\textup{F}^{r}\circ \phi_{*}(\KK)[r] \ar^-{+1}[r]&}.
\]

By Proposition \ref{flag_t_exactness}, $\textup{Gr}_\textup{F}^{r}\circ \phi_{*}[r]$ is $t$-exact, and hence $\textup{Gr}_\textup{F}^{r}\circ \phi_{*}[r]$ is right $t$-exact for $r \leq 0$. Hence for $r \leq -1$, $\textup{F}^{r}\circ \phi_{*}[r]$ is right $t$-exact if $\textup{F}^{r+1}\circ \phi_{*}[r]$ is so. By descending induction on $r$ we are reduced to the case $r=0$, in which case it follows from Proposition \ref{flag_t_exactness}.

Now we show the right $t$-exactness of $\cR_{d} \circ \phi_{*}$. The statement is trivially true for $d=0$. Now suppose $d \geq 1$. Thus it is enough to show that $f_d^{\dagger}\cR_{d} \circ \phi_{*}$ is right $t$-exact whenever $f_{d-1}^{\dagger}\cR_{d-1} \circ \phi_{*}$ is so. As before we have a Cartesian diagram
\begin{equation}\label{Grass_Rad_Corollary_diagram2}
\begin{tikzcd}[row sep=normal, column sep=normal]
   \overline{Q}_{-N+d} \arrow[r,hook] \arrow[d] & \bbP \times_S \textup{Fl} \arrow[d,"1\times_S f_d"] \arrow[r] & \textup{Fl} \arrow[d,"f_d"] \\
    Q_{d} \arrow[r,hook,"i_{d}"] & \bbP \times_S \bG(d) \arrow[r] & \bG(d).
\end{tikzcd}
\end{equation}

Hence using the isomorphism in (\ref{flag_t_exactness_iso}), for any sheaf $K$ on $Z$ and $r=d-N+1\leq 0$, we obtain an exact triangle 

\[
 \small \xymatrix{\textup{Gr}_\textup{F}^{r}\circ \phi_{*}(\KK) \ar[r] & f_d^{\dagger}(\cR_{d} \circ \phi_{*}(\KK))\ar[r] & f_{d-1}^{\dagger}(\cR_{d-1} \circ \phi_{*}(\KK)) \ar^-{+1}[r]&}.
\]

We have already shown that $\textup{Gr}_\textup{F}^{r}\circ \phi_{*}$ is right $t$-exact (for $r \leq 0$) and thus we are done. \end{proof}

\begin{rem}\label{duality_remark_GRass}
    Using \S \ref{basic_prop_BR}, (a) we can now deduce the dual statements. More precisely $R_{d*}\circ \phi_{!}[N-d-1]$ is right $t$-exact when $\phi$ is quasi-finite. Moreover when $Z/S$ is affine, $R_{d*} \circ \phi_{!}$ and $R_d\circ \phi_{!}$ are left $t$-exact.
\end{rem}

\subsection{P=Dec(F) on cohomology}\label{P_Dec(F)_Coho}

In this section, we work over $S=\Spec(k)$. For ease of notation, we shall denote $\cR_{d!}$ (resp. $\cR_{d}$) simply by $R_{d!}$ (resp. $R_{d}$). Now we shall prove Theorem \ref{thm:affine-radon1} using Proposition \ref{Grass_Rad_Corollary}.

\begin{proof}\label{proof_thm_affine_Rad}
For $d \leq 0$, \cite[4.2.4]{BBDG18} implies that $\mathrm{H}^{d-N+1+i}(Z, {}^p\tau_{>i}\KK)=0$ and hence $P^{-i}\mathrm{H}^{d-N+1+i}(Z, \KK)=\mathrm{H}^{d-N+1+i}(Z, \KK)$. Moreover when $Z$ is affine, Artin vanishing implies that $\mathrm{H}^{d-N+1+i}(Z, {}^p\tau_{\leq i}\KK)=0$ for $d \geq N$ and hence $P^{-i}\mathrm{H}^{d-N+1+i}(Z, \KK)=0$. Thus Theorem \ref{thm:affine-radon1} is trivially true when $d\leq0$ or $d \geq N$, and consistent with our convention for $\bG(d)$. Now we assume $0<d<N$.

Applying perverse cohomology functor to the triangle (\ref{basic_triangle}) for the sheaf $\phi_*\KK$ and $\phi_*{}^p\tau_{ \leq i}K$, we get a commutative diagram of perverse sheaves whose rows and left column are exact

   \begin{equation*}
       \xymatrix{  \cpH^i(\cR_{d!}[d-N+1](\phi_*{}^p\tau_{>i}K)) & & \\
                   \cpH^i(\cR_{d!}[d-N+1](\phi_*K)) \ar[r] \ar[u] & \cpH^i(\pi^{\vee}_{d*}\pi_d^{\dagger}[d-N+1]\phi_*K) \ar^-{\alpha_i}[r] & \cpH^{i+1}(\cR_d(\phi_*K)) \\
                   \cpH^i(\cR_{d!}[d-N+1](\phi_*{}^p\tau_{ \leq i}K)) \ar[r] \ar[u] & \cpH^i(\pi^{\vee}_{d*}\pi_d^{\dagger}[d-N+1]\phi_*{}^p\tau_{ \leq i}K) \ar[r] \ar^-{\beta_i}[u] & \cpH^{i+1}(\cR_d(\phi_*{}^p\tau_{ \leq i}K)) \ar[u] }.
   \end{equation*}

Now Proposition \ref{Grass_Rad_Corollary} implies that the perverse sheaf $\cpH^i(\cR_{d!}[d-N+1](\phi_*{}^p\tau_{>i}K))$ vanishes, and thus $\textup{ker}(\alpha_i) \subseteq \textup{Im}(\beta_i)$. Moreover when $Z/S$ is affine, $\cpH^{i+1}(\cR_d(\phi_*{}^p\tau_{ \leq i}K))$ vanishes too and hence $\textup{ker}(\alpha_i)=\textup{Im}(\beta_i)$. Now let $V \subseteq \bG(d)$ be the locus where 
\begin{enumerate}[(a)]
    \item $\cR_{d}(\phi_*K)$ is lisse.
    \item The stalk of $\cR_d(\phi*K)[-d(N-d)]$ at any closed point $v \in V$ is isomorphic to $R\Gamma(\phi^{-1}(\Lambda_v),\KK|_{\phi^{-1}(\Lambda_v)})$.
\end{enumerate}

By constructibility of $\cR_{d}(\phi_*K)$ and generic base change \cite[Chapitre 7. Th\'eor\`eme 1.9]{SGA4.5}, $V$ contains a dense open subset and we replace $V$ with that open subset. Since restriction to an open subset is $t$-exact we have $\textup{ker}(\alpha_i|_{V})\subseteq\textup{Im}(\beta_i|_{V})$ (resp.~$\textup{ker}(\alpha_i|_{V})=\textup{Im}(\beta_i|_{V})$ as the case may be). By smooth base change and Lemma \ref{perverse_usual}, $\textup{Im}(\beta_i|_{V})$ is the constant perverse sheaf with values in $P^{-i}\mathrm{H}^{i+d-N+1}(Z, \KK)$. On the other hand by (a) and (b) above, $\textup{ker}(\alpha_i|_{V})$ is the constant perverse sheaf with values in $\ker \left(\mathrm{H}^{i+d-N+1}(Z, \KK) \to \mathrm{H}^{i+d-N+1}(\phi^{-1}(\Lambda_v),\KK|_{\phi^{-1}(\Lambda_v)}\right)$ for any $v$ in $V$.
\end{proof}

\begin{rem}\label{main_thm_remark}
The argument above implies that $\textup{ker}(\alpha_i)\subseteq\textup{Im}(\beta_i)$ (resp. $\textup{ker}(\alpha_i)=\textup{Im}(\beta_i)$) remains valid even when the base is not necessarily a field.
\end{rem}
\subsection{The case $d=N-1$}\label{d_N-1}

We continue using the notations from the earlier sections. Now we look at the special case of $d=N-1$. We denote the Grassmannian $\bG(N-1)$ by $\bbP^{\vee}$. Under this assumption, Proposition \ref{Grass_Rad_Corollary} and Remark \ref{duality_remark_GRass} imply the following\footnote{Corollary \ref{SWL_Rad} is equivalent to the exactness of $\textup{Gr}_F^0$ in Proposition \ref{flag_t_exactness}, and thus can be readily deduced directly from Lemma \ref{criterion_t_exactness_2} as in the proof of Proposition \ref{flag_t_exactness}.}.

\begin{cor}\label{SWL_Rad}
Let $\phi \colon Z \to \bbP$ be a quasi-finite morphism. Then $\cR_{N-1!} \circ \phi_*$ is left $t$-exact and  $\cR_{N-1*} \circ \phi_!$ is right $t$-exact. Moreover, if $Z/S$ is affine then they are both $t$-exact.
\end{cor}

For the next two results, we work over $S=\Spec(k)$ and use notations from \S \ref{P_Dec(F)_Coho}.

\begin{cor}\label{Deligne_WL}
    Let $\phi: Z \to \bbP$ be a quasi-finite. Let $\KK$ be an upper semi-perverse sheaf on $Z$. Then the cohomology groups $\mathrm{H}^i(Z,j_{!}j^!\KK)=0$ vanish for $i < 0$. Here $j \colon V \hookrightarrow Z$ is the inverse image (under $\phi$) of a generic hyperplane section complement. 
\end{cor}

\begin{proof}
    By constructibility and Deligne's generic base change, we may choose an open dense subset $V \hookrightarrow \bbP^{\vee}$ such that $\cR_{N-1!}(\phi_*\KK)$ is lisse and the stalk of $\cR_{N-1!}(\phi_*\KK)[-N]$ at a closed point $v \in V$ is isomorphic to $R\Gamma(Z,j_{v!}j_v^!\KK)$. Here $j_v \colon Z \backslash \phi^{-1}\Lambda_v \hookrightarrow Z$, is the hyperplane complement. Now suppose $\KK$ was upper semi-perverse on $Z$, then by Corollary \ref{SWL_Rad}, $\cpH^{i}(\cR_{N-1!}\circ\phi_*(\KK))=0$ for $i <0 $ and this vanishing continues to hold when restricted to $V$. Thus by Lemma \ref{perverse_usual}, $\mathrm{H}^i(Z,j_{v!}j_v^!\KK)=0$, for $i <0$. 
\end{proof}

\subsubsection{Proof of Theorem \ref{SWL_Intro}}\label{proof_SWL}

\begin{proof}
The argument for (b) is similar to Corollary \ref{Deligne_WL}, where one has exactness of $\cR_{N-1!} \circ \phi_*$ (as opposed to only left $t$-exactness of $\cR_{N-1!}\circ \phi_*$), thanks to $Z/k$ being affine. Now shall we prove (a).

Let $\cR_{\bbP^{\vee},!} \colon D^{b}_c(\bbP \times_k \bbP^{\vee}) \to D^b_c(\bbP^{\vee} \times_k \bbP^{\vee})$ be the !-Radon transform over the base $\bbP^{\vee}$ (see Equation (\ref{!Radon})). 
Similarly let $\cR_{\bbP,*} \colon D^{b}_c(\bbP \times_k \bbP) \to D^b_c(\bbP \times_k \bbP^{\vee})$ be the *-Radon transform over the base $\bbP$ (see Equation (\ref{*Radon})).

Let $j_{N-1} \colon U_{N-1} \hookrightarrow \bbP \times_k \bbP^{\vee}$ be the complement of the universal hyperplane section (see \S \ref{Grassmannians},(b)). Let $\Delta \colon \bbP \hookrightarrow \bbP \times_k \bbP$ be the diagonal embedding. It follows from the definition of $\cR_{\bbP,*}$, that 
\[
\cR_{\bbP,*} \circ \Delta_*=j_{N-1*}j_{N-1}^*\cR_{\bbP,*} \circ \Delta_*,
\]

\noindent and thus

\[
\cR_{\bbP^{\vee},!} \circ \cR_{\bbP,*} \circ \Delta_*=(\cR_{\bbP^{\vee},!} \circ j_{N-1*}) \circ \left (j_{N-1}^* \circ \left (\cR_{\bbP,*} \circ \Delta_* \right ) \right ).
\]

Now Corollary \ref{SWL_Rad} implies that $\cR_{\bbP,*} \circ \Delta_*$ is $t$-exact. Moreover $j_{N-1}$ is obviously quasi-finite and since the composite $U_{N-1} \hookrightarrow \bbP \times_k \bbP^{\vee} \to \bbP^{\vee}$ is affine, Corollary \ref{SWL_Rad} also implies that $\cR_{\bbP^{\vee},!} \circ j_{N-1*}$ is $t$-exact. Hence implying the $t$-exactness of $\cR_{\bbP^{\vee},!} \circ \cR_{\bbP,*} \circ \Delta_*$\footnote{This statement (and the proof) is also valid over a base $S$.}.

Thus for $\phi$ a quasi-finite and affine morphism, $\cR_{\bbP^{\vee},!} \circ \cR_{\bbP,*} \circ \Delta_* \circ \phi_*$ is $t$-exact. As before using constructibility and generic base change (on $\bbP^{\vee} \times_k \bbP^{\vee}$) for $\cR_{\bbP^{\vee},!} \circ \cR_{\bbP,*} \circ \Delta_* \circ \phi_*\KK$ , we get (a).

\end{proof}

\subsection{Radon transform from $\cP(Z)$ to $\cP(\bbP^{\vee})$ is faithful}\label{inverse_Affine}Suppose we are in the situation of Corollary \ref{SWL_Rad} with $Z/S$ assumed to be affine. Thus $\cR_{N-1!} \circ \phi_*$ is $t$-exact and we can look at the induced functor (also denoted by $\cR_{N-1!} \circ \phi_*$) between $\cP(Z)$ and $\cP(\bbP^{\vee})$. Corollary \ref{SWL_Rad} implies that this is exact. However one can say more thanks to the existence of inverse Radon transform.

The existence of an inverse Radon transform implies that $\cpH^0 \circ \cR_{N-1}$ induces an equivalence of categories between the quotient categories $\cA(\bbP)$\footnote{Here $\cA(\bbP)$ is the quotient of $\cP(\bbP)$ by the Serre sub-category of perverse sheaves coming from $S$ (via $\dagger$-pullback). Similarly for $\cA(\bbP^{\vee})$.} and $\cA(\bbP^{\vee})$ \cite[Proposition 5.7]{Lau}. It follows from the triangle (\ref{basic_triangle}) that $\cpH^0 \circ \cR_!(\bbP^{\vee})=\cpH^0 \circ \cR_{N-1}$ on $\cA(\bbP)$. Thus the former also induces an equivalence of categories. Thus we have functors 

\begin{equation}\label{SWL_Rad_Dia}
    \begin{tikzcd}[column sep=large]
        \cP(Z) \arrow[r,"\phi_*"'] & \cP(\bbP) \arrow[r,"\cpH^0 \circ \cR_{N-1!}"'] \arrow[d,two heads] & \cP(\bbP^{\vee}) \arrow[d,two heads] \\
                                  &  \cA(\bbP) \arrow[r,"\sim","\cpH^0 \circ \cR_{N-1!}"'] & \cA(\bbP^{\vee}).
    \end{tikzcd}
\end{equation}

It follows from Corollary \ref{SWL_Rad} that $\cpH^0 \circ \cR_{N-1!} \circ \phi_*= \cR_{N-1!} \circ \phi_*$. Consider the diagram
\begin{equation}\label{SWL_Rad_Dia2}
  \begin{tikzcd}
      Z \arrow[r,"\phi"] \arrow[dr,"\pi_Z"'] & \bbP \arrow[d,"\pi_{\bbP}"]. \\
                          & S
  \end{tikzcd} 
\end{equation}
Let $\KK$ be a perverse sheaf such that $\cR_{N-1!} \circ \phi_*\KK$ belongs to Serre sub-category of perverse sheaves coming from $S$ on $\bbP^{\vee}$, then diagrams (\ref{SWL_Rad_Dia}) and (\ref{SWL_Rad_Dia2}) imply that 
\[
\phi_*\KK \simeq \pi_{\bbP}^{\dagger}\textup{L}
\]
\noindent for some perverse sheaf $\textup{L}$ on $\bbP$. However, by projection formula implies that 
\[
\pi_{Z*}\KK \simeq \oplus_{i=0}^{N}\textup{L}[-2i+N].
\]

Since by assumption $\pi_Z$ is affine and $\KK$ perverse, Artin vanishing implies $\pi_{Z*}\KK$ is semi-perverse and thus $\textup{L}=0$ and hence so is $\KK$. Thus for $\phi \colon Z \to \bbP$ be quasi-finite and $Z/S$ affine we have proved the following.
\begin{prop}\label{SWL_converse}
    The functors $\cR_{N-1!}\circ \phi_*$, $\cpH^0 \circ \cR_{N-1} \circ \phi_*$, $\cpH^0 \circ \cR_{N-1} \circ \phi_!$ and $\cR_{N-1*} \circ \phi_!$ from $\cP(Z)$ to $\cA(\bbP^{\vee})$ (and hence also to $\cP(\bbP^{\vee})$) are all faithful and exact.   
\end{prop}

\begin{rem}\label{SWL_Converse_Rem}
 Corollary \ref{SWL_Rad} and Proposition \ref{SWL_converse} together imply that a sheaf $\KK$ on $Z$ is perverse iff $R_{N-1!}\circ \phi_*(\KK)$ is so.
\end{rem}

\section{Applications to divisibility of Frobenius eigenvalues and Hodge level}\label{Coho_Div}

In this section we shall apply Corollary \ref{Deligne_WL} to carry out the strategy in \cite{EW22} to obtain a strengthening of Esnault-Katz's cohomological divisibility theorem \cite[Theorem 2.3, (2)]{EK05}. The projective case was handled in \cite[Theorem 1.3]{EW22}. 

\subsection{A Lefschetz theorem for affine Gysin map}
We work over an arbitrary algebraically closed field $k$. Let $\phi \colon Z \to \bbP$ be a quasi-finite map.  Let $\KK$ be a semi-perverse sheaf on $Z$.

\begin{lem}\label{affine_Gysin}
    Then for a general hyperplane section $\Lambda \subset \bbP$ the Gysin map 
    \[
    \rH_c^{r}(\phi^{-1}(\Lambda),i^!\KK) \to \rH_c^{r}(Z,\KK),
    \]
    \noindent is a surjection for $r=1$ and an isomorphism for $r>1$. Here $i \colon \phi^{-1}(\Lambda) \subset  Z$ is the closed embedding.
\end{lem}

\begin{proof}
By the standard Gysin triangle, it suffices to show that $\rH_c^{r}(Z,j_*j^*\KK)$ vanishes for $r>0$. Here $j \colon Z\backslash \phi^{-1}\Lambda \hookrightarrow Z$ is the complement of a general hyperplane section pullback. The statement is dual to the assertion in Corollary \ref{Deligne_WL}.
\end{proof}

Now suppose $i \colon Z \hookrightarrow \bbA^N_k$ be a closed subvariety of dimension $n$. Let $\ell$ be a prime invertible in $k$.
\begin{prop}\label{affine_Gysin2}
    Then for a general hyperplane section $\Lambda \subset \bbA^N$, the Gysin map 
    \[
    \rH^{r-2}_c \left (\Lambda \backslash (\Lambda \cap Z), \bbQ_{\ell} \right )(-1) \to \rH_c^{r}(\bbA^N \backslash Z, \bbQ_{\ell}),
    \]
    \noindent is a surjection for $r-1>\textup{dim}(Z)$.
\end{prop}

\begin{proof}
    If $Z=\emptyset$ we are done. We now assume that $\textup{dim}(Z) \geq 0$. We may also assume $r<2N$, otherwise the result is obvious. For any hyperplane section $\Lambda \subset \bbA^N$, we have a commutative diagram with surjective vertical arrows
\begin{equation}
    \begin{tikzcd}
        \rH_c^{r-1}(\Lambda \cap Z, i^{!}\bbQ_{\ell}) \arrow[r,"\alpha"] \arrow[d,two heads] & \rH^{r-1}_c(Z,\bbQ_{\ell}) \arrow[d,two heads] \\
        \rH^{r}_c \left (\Lambda \backslash (\Lambda \cap Z), i^!\bbQ_{\ell}\right ) \arrow[r,"\beta"] & \rH_c^{r}(\bbA^N \backslash Z, \bbQ_{\ell}). 
    \end{tikzcd}
\end{equation}
Applying Lemma \ref{affine_Gysin} to $\KK=\mathbb{Q}_{\ell}[\textup{dim}(Z)]$ on $Z$, we conclude that for a general hyperplane section $\Lambda$, the map $\alpha$ and thus the map $\beta$ is a surjection for $r-1>\textup{dim}(Z)$.

\end{proof}

\begin{rem}\label{affine_Gysin3}
   Proposition \ref{affine_Gysin2} was already observed for $Z \subset \bbA^N$, a local complete intersection (and thus $\bbQ_{\ell}[\textup{dim}(Z)]$ is perverse) in \cite{EW22} using Deligne's weak Lefschetz Theorem \cite[Corollary A.5]{Kat93}. 
\end{rem}

\subsection{Bounds on cohomological divisibility}
Let $(N;d_1,d_2 \cdots,d_r) \in \bbN \times \bbN^{r}$ be a tuple of natural numbers. For each non-negative integer $j$ consider the following function on $\bbN \times \bbN^{s}$

\begin{equation}\label{BCH_1}
\mu_{j}(N;d_1,d_2 \cdots,d_r)=j+\max(0,\left\lceil\frac{N-j-\sum_id_i}{\max_i d_i}\right\rceil)\geq j.
\end{equation}

Let $\bbF_{q}$ be a finite field with $q$ elements. Let $Z_0 \subset \bbA^N_{\bbF_{q}}$ be an affine scheme defined by the vanishing of $r$ polynomials of degree $d_1,d_2 \cdots, d_r$ with $N,r \geq 1$. We denote by $Z$ (resp. $\bbA^N$), the base change of $Z_0$ (resp. $\bbA^N_{\bbF_q}$) to an algebraic closure (say $\bar{\bbF}_q$) of $\bbF_q$. 

\subsubsection*{Summary of previous results on cohomological divisibility:}Let $\rH_c^*(\hspace{0.3cm})$ denote the compactly supported $\ell$-adic cohomology, for $\ell \nmid q$. By \cite[XXI, Appendice]{SGA7II}, the eigenvalues of the geometric Frobenius on $\rH_c^*(Z)$ (and $\rH^*_c(\bbA^N\backslash Z)$) are algebraic integers and thus it makes sense to talk about their divisibility properties. 
 It follows from \cite[XXI, Appendice, (5.2)]{SGA7II} that the Frobenius eigenvalues on $\rH_c^{N+j}(\bbA^N \backslash Z)$ are divisible by $q^j$\footnote{This holds for any separated scheme of finite type over $\bbF_q$.} for $j \geq 0$. In \cite[Theorem 2.3, (2)]{EK05}, Esnault-Katz showed an improved bound replacing $q^j$ above by $q^{\mu_j(N;d_1,d_2 \cdots d_r)}$. We now prove the main result of this section.

\subsubsection{Proof of Theorem \ref{CD_Intro}}\label{CD_Proof}
\begin{proof}
The long exact sequence (with values in $\rH^*_c)$ for the triple $(Z,\bbA^N,\bbA^N \backslash Z)$ together with the Frobenius equivariant isomorphism $R\Gamma_c(\bbA^N,\bbQ_{\ell}) \simeq \bbQ_{\ell}(-N)$ implies that it suffices to prove (b). Moreover, in doing so we may freely extend the field of definition of $Z$ since the geometric Frobenius concerning $\bbF_{q^r}$ is the $r^{\mathrm{th}}$ iterate of the geometric Frobenius with respect to $\bbF_q$.

First suppose $\textup{dim}(Z)=-1$, then $Z=\emptyset$ and in this case the theorem is obvious since $\rH_c^{\textup{dim}(Z)+j+1}(\bbA^N\backslash Z)$ vanishes for $0 \leq j \leq \textup{dim}(Z)+1$. Thus we may assume $\textup{dim}(Z) \geq 0$. Next suppose $N=1$, then since $d_i \geq 1$ we must have $\textup{dim}(Z)=0$ and $\mu_j(N;d_1,d_2 \cdots,d_r)=j$. Also note that $\rH_c^{1}(\bbA^1\backslash Z) \simeq \rH_c^{0}(Z) \simeq \oplus^{\#Z(\bar{\bbF}_{q})}\bbQ_{\ell}$ (the $j=0$ case) and $\rH_c^{2}(\bbA^1\backslash Z) \simeq \bbQ_{\ell}(-1)$ (the $j=1$ case). Thus the theorem is verified when $N=1$. In the rest of the proof, we assume $N \geq 2$ and $\textup{dim}(Z) \geq 0$.

Using Proposition \ref{affine_Gysin} and if necessary extending the base field (and still denoting it by $\bbF_q$), we have a hyperplane $\Lambda_0 \subset \bbA_{\bbF_q}^N$ such that the Gysin map

\begin{equation}\label{CD_Eq1}
\begin{tikzcd}
    \rH^{\textup{dim}(Z)+j-1}_c \left (\Lambda \backslash (\Lambda \cap Z)\right )(-1) \arrow[r,two heads] & \rH_c^{\textup{dim}(Z)+j+1}(\bbA^N \backslash Z),
\end{tikzcd}
\end{equation}
\noindent is a surjection for $j>0$ and such that $\textup{dim}(\Lambda \cap Z)=\textup{dim}(Z)-1$. Here $\Lambda$ is the base change of $\Lambda_0$ to $\bar{\bbF}_q$. Note that $\Lambda \cap Z \subset \Lambda$ is defined by $r$ equations of degrees $d_1,d_2 \cdots ,d_r$ in $\Lambda$ an affine space of dimension $N-1\geq 1$.

By induction on $\textup{dim}(Z)$, we may assume that the eigenvalues of the geometric Frobenius acting on  $\rH^{(\textup{dim}(Z)-1)+(j-1)+1}_c \left (\Lambda \backslash (\Lambda \cap Z) \right )$ are divisible by $q^{\mu_{j-1}(N-1;d_1,d_2 \cdots,d_r)}$ for $0 \leq j-1 \leq \textup{dim}(Z)$. Thus the surjection in (\ref{CD_Eq1}) implies that the eigenvalues of the geometric Frobenius acting on $\rH_c^{\textup{dim}(Z)+j+1}(\bbA^N \backslash Z)$ are divisible by $q^{\mu_{j-1}(N-1;d_1,d_2 \cdots,d_r)}+1=q^{\mu_j(N;d_1,d_2 \cdots,d_r)}$ (where we have accounted for the Tate twist on the left), in the range $1 \leq j \leq \textup{dim}(Z)+1$. Finally for $j=0$, the result follows from \cite[Theorem 2.3, (1)]{EK05}.
\end{proof}

An analogous argument can be used to deduce the following result about Hodge levels\footnote{The Hodge level of a Mixed Hodge structure $(\rH_{\bbZ},\textup{W}_{\bullet},\textup{F}^{\bullet})$ is the largest integer $n$ such that $\textup{F}^n\rH_{\bbC}=\rH_{\bbC}$.}, where instead of \cite[Theorem 2.3, (1)]{EK05}  one uses \cite[\S 5.3]{EK05} in the proof of Theorem \ref{CD_Intro} above. In the next theorem $Z \subset \bbA^N_{\bbC}$ is an affine scheme defined by $r$ polynomials of degree $d_1,d_2 \cdots, d_r$ with $N,r \geq 1$.
\begin{thm}\label{HL}
    The Hodge levels of
\begin{enumerate}[(a)]
    \item $\rH_c^{\textup{dim}(Z)+j}(Z)$, for $j\geq 0$ and 
    \item $\rH_c^{\textup{dim}(Z)+j+1}(\bbA^N\backslash Z)$ for $0 \leq j \leq \textup{dim}(Z)+1$
\end{enumerate}
    \noindent are atleast $\mu_j(N;d_1,d_2 \cdots,d_r)$.
\end{thm}

\section{P=Dec(F) as functors}\label{P_Dec_F_Functors}
In this section, we shall upgrade Theorem \ref{thm:affine-radon1} to an equivalence of functors. We shall make use of the results proved in the appendix. We begin by recalling the same

\subsection{Summary of notations and results from Appendix \ref{appendix}}\label{appendixrecall}
As before we fix an algebraically closed field $k$ and let $X/k$ be a separated scheme of finite type over $k$.

\begin{enumerate}[(a)]
    \item  There is a symmetric monoidal stable $\infty$-category $\cD_{\textup{cons}}(X)$ whose underlying homotopy category is $D^b_c(X)$. It comes equipped with a perverse $t$-structure (see \S \ref{D_cons}). 
    \item We denote by $\cD_{\textup{indcons}}(X)$, the ind-completion of $\cD_{\textup{cons}}(X)$. This is a symmetric monoidal presentable stable $\infty$-category (see \S \ref{properties_ind_cons}). Moreover there is an unique $t$-structure on $\cD_{\textup{indcons}}(X)$ such that the inclusion $\cD_{\textup{cons}}(X) \subseteq \cD_{\textup{indcons}}(X)$ is $t$-exact. 
    \item We denote by $\cD F_{\textup{indcons}}(X) :=\textup{Func}(\bbZ^{\textup{op}},\cD_{\textup{indcons}})$, the filtered category of $\cD_{\textup{indcons}}(X)$. This is also a symmetric monoidal presentable stable $\infty$-category. This comes equipped with functors $\textup{Gr}^i$, $\textup{ev}_i$ and $\KK^{\bullet} \to \KK^{\bullet}(-\infty)$ with values in $\cD_{\textup{indcons}}$ (\S \ref{properties_ind_cons}, (b), (c) and (d)).
    \item The category $\cD F_{\textup{indcons}}(X)$ has a Beilinson $t$-structure (see Theorem \ref{Beilinson_t_structure}). An object $\KK^{\bullet}$ in $\cD F_{\textup{indcons}}(X)$ is in $\cD F_{\textup{indcons}}^{\leq 0}$ (resp. $\cD F_{\textup{indcons}}^{\geq 0})$ if $\textup{Gr}^i\KK^{\bullet}$ (resp. $\KK^i$) is in $\cD^{\leq i}_{\textup{indcons}}$ (resp. $\cD^{\geq i}_{\textup{indcons}}$). We denote the truncations for this $t$-structure by $(^{B}\tau_{\leq 0},{}^{B}\tau_{\geq 0})$. 
    \item There is a functor $\textup{Dec} \colon \cD F_{\textup{indcons}}(X) \to \cD F_{\textup{indcons}}(X)$ which on objects is given by $\textup{Dec}(\KK^{\bullet})^{i}=({}^B\tau_{\leq -i}\KK^{\bullet})(-\infty)$.
    \item We let $\cD F_{\textup{cons}}(X)$\footnote{Objects in this category are to be thought of as finite filtered constructible complexes.} to be the full subcategory of $\cD F_{\textup{indcons}}(X)$ consisting of objects $\KK^{\bullet}$ such that
\begin{enumerate}[(i)]
    \item $\KK^{\bullet}$ is complete, that is $\lim \KK^i=0$ and there exists an $N$ such that $\textup{Gr}^{i}(\KK^{\bullet})=0$ for $|i|>N$.
    \item $\textup{Gr}^{i}(\KK^{\bullet})$ belong to the full subcategory $\cD_{cons}(X)$ of $\cD_{\textup{indcons}}(X)$.
\end{enumerate}
    \item By Proposition \ref{t_strtucture_DF_cons}, $\cD F_{\textup{cons}}(X)$ is a symmetric monoidal stable $\infty$-category and the Beilinson $t$-structure from $\cD_{\textup{indcons}}(X)$ restricts to one on $\cD F_{\textup{cons}}(X)$.
    \item Moreover the functor $\textup{Dec}$ from (e) above, preserves the subcategory $\cD F_{\textup{cons}}(X)$ of  $\cD_{\textup{indcons}}(X)$ (see Proposition \ref{t_strtucture_DF_cons}, (b)).
    
\end{enumerate}

\subsection{Revisiting Theorem \ref{thm:affine-radon1}}\label{Revisiting}
We recall the set-up of Proposition \ref{flag_t_exactness}. We have a base scheme $S/k$, a locally free sheaf $\sE$ on $S$ of rank $N+1$. We denote by $\bbP$ (resp. $\textup{Fl}$) the associated projective bundle (resp.~flag bundle).

We also defined functors\footnote{We continue using the same notation for the functor induced between the corresponding $\infty$-categories.} 
\[
\textup{F}^{r}, \textup{Gr}^r_{\textup{F}} \colon \cD_{\textup{cons}}(\bbP) \to \cD_{\textup{cons}}(\textup{Fl}).
\]

We can put the $\textup{F}^r$'s together and view them as a functor 

\[
\textup{F} \colon \cD_{\textup{cons}}(\bbP) \to \cD F_{\textup{cons}}(\textup{Fl}),
\]

\noindent such that for any sheaf $\KK$, $\textup{ev}_r(\textup{F}(\KK))=\textup{F}^r(\KK)$ and $\textup{Gr}^r(\textup{F}(\KK))=\textup{Gr}^r_{\textup{F}}(\KK)$. Thus using Proposition \ref{t_strtucture_DF_cons}, (a) we may restate Proposition \ref{flag_t_exactness} as follows. Let $\phi \colon Z \to \bbP$ be a quasi-finite morphism. 

\begin{prop}\label{prop_revist}
    The functor $\textup{F} \circ \phi_* \colon \cD_{\textup{cons}}(Z) \to \cD F_{\textup{cons}}(\textup{Fl})$ is left $t$-exact. Moreover, if $Z/S$ is affine then it is $t$-exact.
\end{prop}

In the rest of this section, we assume $\phi \colon Z \to \bbP$ is quasi-finite with $Z/S$ affine. In particular Proposition \ref{prop_revist} implies that there is an equivalence of functors
\begin{equation}\label{prop_revisit_1}
    \textup{F} \circ \phi_* \circ {}^p\tau_{\leq -r} \simeq {}^B\tau_{\leq -r}\circ\textup{F}\circ\phi_*.
\end{equation}

Let $\overline{\pi}$ (resp. $\overline{\pi}^{\vee}$) denote the projection from $\bbP \times_S \textup{Fl}$ to $\bbP$ (resp.~$\textup{Fl}$). We now define a functor  
\[
\textup{P} \colon \cD_{\textup{cons}}(\bbP) \to \cD F_{\textup{cons}}(\textup{Fl}),
\]
\noindent such that for any sheaf $\KK$, $\textup{ev}_r(\textup{P}(\KK)):=\overline{\pi}^{\vee}_{*}\overline{\pi}^{\dagger}
{}^p\tau_{\leq -r}(\KK)$ and the obvious maps between them. Hence $\textup{Gr}^r(\textup{P}(\KK))=\cpH^{-r}(\KK)[r]$.

Combining Proposition \ref{prop_revist}, Equation (\ref{prop_revisit_1}) and the definition of the operator \textup{Dec} (see \S \ref{appendixrecall}, (e) above) we have the following result.

\begin{thm}\label{P_Dec_F_Functor}
  There is a natural equivalence of functors
  \[
  \mathrm{P} \circ \phi_* \simeq \mathrm{Dec}(\mathrm{F}\circ\phi_*).
  \]
\end{thm}

An immediate corollary to Theorem \ref{P_Dec_F_Functor} and Corollary \ref{Beilinson_t_structure_Gr_i_Cor}, (a) is the following comparison of associated spectral sequences \cite[Proposition 1.2.2.14]{HA} with values in perverse sheaves on $\textup{Fl}$. 

\begin{cor}\label{iso_spectral_sequence}
   For any sheaf $\KK$ on $Z$, there is a natural (in $\KK$) isomorphism\footnote{The isomorphism of the spectral sequences could also have been obtained using Postnikov systems at the level of derived categories, a detour to the $(\infty,1)$-world is not necessary.} between the constant perverse spectral sequence
   \[
   E_1^{i,j}=\cpH^{i+j}(\overline{\pi}^{\vee}_{*}\overline{\pi}^{\dagger}\cpH^{-j}(\phi_*\KK)[j]) \implies \cpH^{i+j}(\overline{\pi}^{\vee}_{*}\overline{\pi}^{\dagger}\phi_*\KK)
   \]
   \noindent and the  d\'ecal\'ee of the universal flag filtration spectral sequence
   \[
   E_1^{i,j}=\cpH^{i+j}(\textup{Gr}^{-j}_{\textup{F}}\phi_*\KK[j]) \implies \cpH^{i+j}(\overline{\pi}^{\vee}_{*}\overline{\pi}^{\dagger}\phi_*\KK).
   \]
   
\end{cor}

\appendix 

\section{Beilinson $t$-structure on $\cD F_{\textup{cons}}(X)$}\label{appendix}

Throughout this appendix, we work over an algebraically closed field $k$ and a scheme $X/k$ a separated scheme of finite type. We also fix a coefficient ring $\Lambda$, a finite ring with torsion invertible in $k$ or an algebraic extension $E/\bbQ_{\ell}$.

By $D^b_c(X)$ we shall mean the bounded constructible derived category of \'etale sheaves on $X$ with coefficients in $\Lambda$. This comes equipped the standard \cite[1.1.2, (e)]{Del80} and perverse $t$-structure \cite[\S 4.0]{BBDG18}. 

\subsection{A first attempt}\label{first_attempt}
Our aim in this appendix is to construct a filtered version of $D^b_c(X)$ which meets the following three conditions

\begin{enumerate}[(a)]
    \item It has a lift of the perverse $t$-structure from $D^b_c(X)$. Let us call the truncation maps in this $t$-structure $(^{B}\tau_{\leq 0},{}^{B}\tau_{\geq 0})$.
    \item Objects in the filtered category should correspond to diagrams
    \[
    \KK^{\bullet}:=\{\cdots \to \KK^{i} \to \KK^{i-1} \to \KK^{i-2} \cdots\},
    \]
    \noindent where $\KK^i$'s are sheaves, thought of as the $i^{\mathrm{th}}$ level in the filtration.
    \item This category comes equipped with an endo functor $\textup{Dec}$, defined as follows
    \[
    \KK^{\bullet} \to \textup{Dec}(\KK^{\bullet})^{i}:={}^{B}\tau_{\leq -i}(\KK^{\bullet})(-\infty)
    \]
    meant to mimic Deligne's décalée \cite[1.3.3]{Del71}.
\end{enumerate}

Successful execution of these conditions would then imply P=Dec(F) at the level of filtered complexes as an immediate corollary to the $t$-exactness of a Brylinski-Radon transform.

As a first attempt one could begin with the filtered derived category $DF^b_c(X)$ \cite[1.1.2, (d)]{Del80}, whose construction is already quite delicate for $\Lambda=\bbQ_{\ell}$. One remedy would be to use the pro-\'etale site \cite{BS15} together with results of Schapira-Schneiders \cite{Schapira-Schneiders} which would naturally lead us to consider the derived category of the abelian category of covariant functors, $\textup{Func}(\bbZ^{\textup{op}},\textup{Sh}(X_{\textup{pro\'et}}))$\footnote{Here $\bbZ^{\textup{op}}$ is the category whose underlying poset is $\bbZ^{\textup{op}}$,}. Together with Beilinson's construction of f-categories \cite[Defintion A.1]{Bei87}, this would (almost meet) our conditions (a) and (b)\footnote{We say almost because one would have liked to consider not the derived category of $\textup{Func}(\bbZ^{\textup{op}},\textup{Sh}(X_{\textup{pro\'et}}))$, but rather $\textup{Func}(\bbZ^{\textup{op}},D^b_c(X_{\textup{pro\'et}}))$.}. This leaves us with (c). There are at least a couple of issues here
\begin{enumerate}[(i)]
    \item Already for $X=\Spec(k)$, $D^b_c(X)$ does not have arbitrary (even sequential colimits). A partial remedy for this would be restricting the filtrations we allow in (b) above.
    \item Even with the restricted class of filtrations, it is unclear in what sense $\mathrm{Dec}$ is a functor. $\mathrm{Dec}(\KK^{\bullet})$ is defined using the $t$-structure on the derived category (which in turn is a lift of the perverse $t$-structure again defined on $D^b_c(X)$).
\end{enumerate}

\subsubsection{Strategy of proof}To remedy the above issues we mimic the approach of \cite[\S 5.1]{BMS19} where this was carried out for $D(R)$, the stable $\infty$ category of $R$-modules. In particular, we would like an equivalent of their Theorem 5.4, which as the authors point out is a transport of Beilinson's construction of f-categories to $D(R)$.

In the rest of the appendix we shall work with $\infty$-categories \cite[\S 1.1.2]{HTT} and in particular stable $\infty$-categories \cite[Definition 1.1.9]{HA}. We will aim to construct a symmetric monoidal stable $\infty$-category denoted by $\cD F_{\textup{indcons}}(X)$ which will satisfy (a)-(c) from \S \ref{first_attempt}. Our desired category $\cD F_{\textup{cons}}(X)$ will be a full subcategory of $\textup{Func}(\bbZ^{\textup{op}},\cD_{\textup{indcons}}(X))$\footnote{In the rest of the appendix unless specified ordinary categories are to be thought of as $\infty$-categories \cite[Proposition 1.1.2.2]{HTT}.}, where $\cD_{\textup{indcons}}(X)$ is itself a symmetric monoidal \textit{presentable} stable $\infty$-category.

To achieve our goal of constructing a category that satisfies (a)-(c) above we need to 
\begin{enumerate}[(i)]
    \item define the category $\cD_{\textup{indcons}}(X)$ and equip it with a perverse $t$-structure\footnote{A $t$-structure on a stable $\infty$-category is by definition a $t$-structure on its homotopy category, which is a triangulated category (see \cite[Definition 1.2.1.4]{HA} ).},
    \item lift this $t$-structure to the functor category $\textup{Func}(\bbZ^{\textup{op}},\cD_{\textup{indcons}}(X))$ and,
    \item show that the $t$-structure restricts to one on $\cD F_{\textup{cons}}(X)$.
\end{enumerate}

\subsection{The perverse $t$-structure on $\cD_{indcons}(X)$}
In this section, we will deal with (i) above, and lift the $t$-structure to $\cD F_{\textup{indcons}}(X)$ in the next section. As before $X/k$ is a separated scheme of finite type and $\Lambda$ is a coefficient ring. 

\subsubsection{The category $\cD_{cons}(X)$}\label{D_cons} In \cite[Definition 1.1]{HRS23} the authors construct a symmetric, monoidal and stable $\infty$-category $\cD_{\textup{cons}}(X)$ with the following properties.

\begin{enumerate}[(a)]
    \item Their underlying homotopy categories agree with the usual derived category $D^b_c(X)$ \cite[Proposition 7.1, Theorem 7.7, Lemma 7.9]{HRS23}. In particular the categories $\cD_{\textup{cons}}(X)$ come equipped with a perverse $t$-structure $({}^{p}\cD_{\textup{cons}}^{\leq 0},{}^p\cD_{\textup{cons}}^{\geq 0})$.
    \item The usual six-functor formalism extends to the assignment $X \to \cD_{\textup{cons}}(X)$ \cite[Remark 7.10]{HRS23}.
\end{enumerate}

\subsubsection{The category $\textup{Ind}(\cD_{cons}(X))$}\label{properties_ind_cons}
Let $\textup{Ind}(\cD_{cons}(X))$ be the Ind-completion \cite[Lemma 5.3.2.9]{HA} of $\cD_{cons}(X)$. 
\begin{enumerate}[(a)]
    \item This category is symmetric monoidal presentable stable $\infty$-category. We shall denote this by $\cD_{\textup{indcons}}(X)$ (see \cite[Corollary 8.3]{HRS23} for why this notation is reasonable).
    \item $\cD_{\textup{cons}}(X)$ is a full subcategory of $\cD_{\textup{indcons}}(X)$ and consists precisely of the compact objects \cite[Proposition 8.2]{HRS23}.
    \item The perverse $t$-structure on $\cD_{\textup{cons}}(X)$ extends uniquely to one on $\cD_{\textup{indcons}}(X)$, denoted by $({}^{p}\cD_{\textup{cons}}^{\leq 0},{}^p\cD_{\textup{cons}}^{\geq 0})$ \cite[Lemma 6.1.2, (a)]{BKV22}. We call this the perverse $t$-structure on $\cD_{\textup{indcons}}(X)$.
    \item Moreover $^{p}\cD_{\textup{indcons}}^{\leq 0}=\textup{Ind}(^{p}\cD_{\textup{cons}}^{\leq 0})$, $^{p}\cD_{\textup{indcons}}^{\geq 0}=\textup{Ind}(^{p}\cD_{\textup{cons}}^{\geq 0})$. Further $^{p}\cD_{\textup{Indcons}}^{\leq 0}$ is closed under colimits and $^{p}\cD_{\textup{indcons}}^{\geq 0}$ is closed under limits \cite[Lemma 6.1.2, (a), (b)]{BKV22}. 
\end{enumerate}

\subsection{The category $\cD F_{\textup{indcons}}(X)$}
We continue using the notation from previous sections. We define 

\[
\cD F_{\textup{indcons}}(X):=\textup{Func}(\bbZ^{\textup{op}},\cD_{\textup{indcons}}(X)).
\]
We shall denote an object in $\cD F_{\textup{indcons}}(X)$ by $\KK^{\bullet}$.
\subsubsection{Some basic operations on $\cD F_{\textup{indcons}}(X)$}\label{operations_Filtered}
\begin{enumerate}[(a)]
    \item $\cD F_{\textup{indcons}}(X)$ is a symmetric monoidal presentable stable $\infty$-category \cite[\S 2.23]{GP18}.
    \item For any integer $i \in \bbZ$, there is a functor $\textup{ev}_{i} \colon \cD F_{\textup{indcons}}(X) \to \cD_{\textup{indcons}}(X)$, which on objects will be denoted by $\KK^{\bullet} \to \KK^i$.
    \item For any integer $i \in \bbZ$, there is a functor $\textup{Gr}^{i} \colon \cD F_{\textup{indcons}}(X) \to \cD_{\textup{indcons}}(X)$, which on objects is given by $\KK^{\bullet} \to \textup{cofib}(\KK^{i+1} \to \KK^{i})$.
    \item We also have a functor sending $\KK \to \KK(-\infty):=\colim \KK^{i}$.
    \item An object $\KK^{\bullet}$ is said to be \textit{complete} if $\lim_i \KK^{i}=0$.
    \item The $\textup{ev}_{i}$ and $\textup{Gr}^{i}$ functors commute with all limits and colimits \cite[\S 2]{GP18}. In particular, they are exact.
    \item Finally we note that $\textup{ev}_i$ has a left adjoint, denoted by $\textup{ins}_{i}$, which sends an object $\mathrm{L} \to \textup{ins}_{i}(\mathrm{L})^{\bullet}$ with $(\textup{ins}_{i}(\mathrm{L}))^{j}=\mathrm{L}$ for $j \leq i$ and $0$ otherwise. 
    \item The functors $\{\textup{Gr}^{i}\}_{i \in \bbZ}$ form a conservative family of functors on complete objects \cite[\S 2.7]{GP18}.
    
\end{enumerate}

Now define full subcategories $({}^p\cD F_{\textup{indcons}}^{\leq 0},{}^p\cD F_{\textup{indcons}}^{\geq 0})$ as follows

\[
^p\cD F_{\textup{indcons}}^{\leq 0}:=\{\KK^{\bullet};\textup{Gr}^{i}(\KK^{\bullet}) \in{}^{p}\cD_{\textup{indcons}}^{\leq i}\},
\]
\noindent and dually

\[
^p\cD F_{\textup{indcons}}^{\geq 0}:=\{\KK^{\bullet};\KK^{i} \in{}^{p}\cD_{\textup{indcons}}^{\geq i}\}.
\]

In the next section, we shall prove the following result mimicking the proof in \cite[Theorem 5.3]{BMS19}, where it is proved for the filtered $\infty$-category of $R$-modules.

\begin{thm}(Beilinson)\label{Beilinson_t_structure}
    The subcategories $({}^p\cD F_{\textup{indcons}}^{\leq 0},{}^p\cD F_{\textup{indcons}}^{\geq 0})$ define a $t$-structure on $\cD F_{\textup{indcons}}(X)$.
\end{thm}

\subsection{Proof of Theorem \ref{Beilinson_t_structure}}\label{proof_Beilinson_t_Structure}
\begin{proof}
    We begin by noting that $^{p}\cD_{\textup{Indcons}}^{\leq i}$ is stable under colimits (see \S \ref{properties_ind_cons}, (d)). Moreover $\textup{Gr}^{i}$ commutes with colimits (see \S \ref{operations_Filtered},(c)), and thus by definition, $^p\cD F_{\textup{indcons}}^{\leq 0}$ is stable under colimits. By presentability\footnote{This step prevented us from directly using the more obvious category $\cD F_{\textup{cons}}(X)$.} the inclusion $^p\cD F_{\textup{indcons}}^{\leq 0} \subset{}^p\cD F_{\textup{indcons}}(X)$ has a right adjoint, which we denote by $R \colon{}^p\cD F_{\textup{indcons}}(X) \to{}^p\cD F_{\textup{indcons}}^{\leq 0}$. In particular given any object $\KK^{\bullet}$ in $^p\cD F_{\textup{indcons}}(X)$, we have a fiber sequence
    \begin{equation}\label{proof_Beilinson_t_structure_1}
    R(\KK^{\bullet}) \to \KK^{\bullet} \to Q(\KK^{\bullet}).
    \end{equation}
    \noindent Thus to complete the proof of the theorem it suffices to show that $Q(\KK^{\bullet}) \in{}^p\cD F_{\textup{indcons}}^{>0}$ or equivalently that the space $\textup{Map}_{\cD_{\textup{indcons}}(X)}(\KK',Q(\KK^{\bullet})^{i})$ is contractible for $\KK' \in{}^p\cD_{\textup{indcons}}^{\leq i}$. By adjunction (see \S \ref{operations_Filtered}, (e)) we have 
    \[
    \textup{Map}_{\cD_{\textup{indcons}}(X)}(\KK',Q(\KK^{\bullet})^{i})=\textup{Map}_{\cD F_{\textup{indcons}}(X)}(\textup{ins}_{i}\KK',Q(\KK^{\bullet})),
    \]
\noindent and thus it is enough to show that $\textup{Map}_{\cD F_{\textup{indcons}}(X)}(\textup{ins}_{i}\KK',Q(\KK^{\bullet}))$ is contractible for any $\KK' \in{}^p\cD_{\textup{indcons}}^{\leq i}$. Consider the diagram below obtained as a pullback of the fibre sequence in Diagram (\ref{proof_Beilinson_t_structure_1}) along the map $\eta \in \textup{Map}_{\cD F_{\textup{indcons}}(X)}(\textup{ins}_{i}\KK',Q(\KK^{\bullet}))$

\begin{equation}\label{proof_Beilinson_t_structure_2}
    \begin{tikzcd}
        R(\KK^{\bullet}) \arrow[r] \arrow[d,equal] & \textup{L}^{\bullet} \arrow[r] \arrow[dl,dotted] \arrow[d] &\textup{ins}_i(\KK') \arrow[d,"\eta"] \\
        R(\KK^{\bullet}) \arrow[r] & \KK^{\bullet} \arrow[r] & Q(\KK^{\bullet}).
    \end{tikzcd}
\end{equation}

Recall that $\KK' \in{}^p\cD_{\textup{indcons}}^{\leq i}$, hence $\textup{ins}_i(\KK')$ is in ${}^p\cD F_{\textup{indcons}}^{\leq 0}$ by definition of the functor $\textup{ins}_i(\KK')$. This implies $\mathrm{L}^{\bullet}$ is also in $^p\cD F_{\textup{indcons}}^{\leq 0}$ (since $R(\KK^{\bullet}) \in{}^p\cD F_{\textup{indcons}}^{\leq 0}$ by construction and ${}^p\cD F_{\textup{indcons}}^{\leq 0}$ is closed under extensions). Since $R(\KK^{\bullet}) \to \KK^{\bullet}$ is the counit of the adjunction and $\mathrm{L}^{\bullet} \in {}^p\cD F_{\textup{indcons}}^{\leq 0}$, there exists a dotted arrow in Diagram (\ref{proof_Beilinson_t_structure_2}) which provides a splitting for the top row. This implies that $\eta$ is $0$ (in the homotopy category) as required.

\end{proof}

\begin{nota}\label{Beilinson_t_structure_notat}
 We shall denote the truncations for the $t$-structure induced on $\cD F_{\textup{indcons}}(X)$ via Theorem \ref{proof_Beilinson_t_Structure} by $(^{B}\tau_{\leq 0},{}^{B}\tau_{\geq 0})$.
\end{nota}

\begin{lem}\label{Beilinson_t_structure_Gr_i}
The functors $\textup{Gr}^{i}[i] \colon \cD F_{\textup{indcons}}(X) \to \cD_{\textup{indcons}}(X)$ are $t$-exact for the Beilinson $t$-structure on the source and perverse $t$-structure on the target.
\end{lem}

\begin{proof}
That $\textup{Gr}^{i}[i]$ is right $t$-exact is clear from the definition. For the left $t$-exactness note that for any $\KK^{\bullet}$ there is a fibre sequence $\KK^{i}[i] \to \textup{Gr}^{i}\KK^{\bullet}[i] \to \KK^{i+1}[i+1]$, and hence the left $t$-exactness follows from the fact that $\cD^{\geq 0}_{\textup{indcons}}$ is closed under extensions.
\end{proof}

\begin{cor}\label{Beilinson_t_structure_Gr_i_Cor}
We have the following corollary to Lemma \ref{Beilinson_t_structure_Gr_i}.
    \begin{enumerate}[(a)]
        \item There is a natural isomorphism $\textup{Gr}^{i}\circ{}^{B}\tau_{\leq j} \simeq{}^p\tau_{\leq i+j}\textup{Gr}^{i}$ (See \cite[Theorem 5.4, (2)]{BMS19}).
        \item A complete object $\KK^{\bullet}$ in $\cD F_{\textup{indcons}}(X)$ is in the heart of the Beilinson $t$-structure if and only if the $\textup{Gr}^{i}\KK[i]$ belong to the heart for all $i \in \bbZ$.
    \end{enumerate}
\end{cor}

\begin{proof}
Statement (a) is immediate from $t$-exactness of $\textup{Gr}^i[i]$ (Lemma \ref{Beilinson_t_structure_Gr_i}).

For (b) observe that the only if the direction is clear from Lemma \ref{Beilinson_t_structure_Gr_i} and is valid without the completeness hypothesis. We now prove the `if' direction of the corollary. For an object $\KK^{\bullet}$, if for all $i$, $\textup{Gr}^{i}\KK^{\bullet}[i]$ belongs to the heart of the perverse $t$-structure on $\cD_{\textup{indcons}}(X)$, then $\KK^{\bullet} \in{}^pD F^{\leq 0}_{\textup{indcons}}$. For the other inclusion observe that by completeness $\KK^i \simeq \lim_{j\geq i}\textup{cofib}(\KK^i \to \KK^j)$, and the result follows from the stability of $\cD^{\geq 0}_{\textup{indcons}}$ under limits (see \S \ref{properties_ind_cons}, (d)). 
\end{proof}

\begin{defn}\label{decalage}
 We define a functor $\textup{Dec} \colon \cD F_{\textup{indcons}}(X) \to \cD F_{\textup{indcons}}(X)$ which on objects is given by 
\begin{equation}\label{def_Dec}
\textup{Dec}(\KK^{\bullet})^{i} = ({}^B\tau_{\leq -i}\KK^{\bullet})(-\infty).
\end{equation}   
\end{defn}

\subsection{The category $\cD F_{\textup{cons}}(X)$}\label{cDF_cons}
Consider the following full subcategory $\cD F_{\textup{cons}}(X)$ of $\cD F_{\textup{indcons}}(X)$ consisting of objects $\KK^{\bullet}$ such that
\begin{enumerate}[(a)]
    \item $\KK^{\bullet}$ is complete and there exists an $N$ such that $\textup{Gr}^{i}(\KK^{\bullet})=0$ for $|i|>N$.
    \item $\textup{Gr}^{i}(\KK^{\bullet})$ belong to the full subcategory $\cD_{cons}(X)$ of $\cD_{\textup{indcons}}(X)$.
\end{enumerate}

We have the following

\begin{prop}\label{t_strtucture_DF_cons}
    $\cD F_{\textup{cons}}(X)$ is a symmetric monoidal stable $\infty$-category and the Beilinson $t$-structure (from $\cD F_{\textup{indcons}}(X)$) restricts to one on $\cD F_{\textup{cons}}(X)$. Moreover
    \begin{enumerate}[(a)]
        \item an object $\KK^{\bullet}$ in  $\cD F_{\textup{cons}}(X)$ is in the heart of the Beilinson $t$-structure iff $\textup{Gr}^{i}[i]\KK^{\bullet}$ are perverse.
        \item The functor $\textup{Dec}$ (see Definition \ref{decalage}) preserves the subcategory $\cD F_{\textup{cons}}(X)$. 
    \end{enumerate}
\end{prop}

\begin{proof}
Conditions (a) and (b) in the definition of $\cD F_{\textup{cons}}(X)$ together imply that,  an object $\KK^{\bullet}$ in $\cD F_{\textup{indcons}}(X)$ belongs to $\cD F_{\textup{cons}}(X)$ iff $\KK^i$'s are in $\cD_{\textup{cons}}(X)$. 

Now since $\cD_{\textup{cons}}(X)$ is itself a symmetric monoidal subcategory of $\cD_{\textup{indcons}}(X)$, the symmetric monoidal part of the claim follows from the definition of day convolution structure \cite[\S 2.23]{GP18} on $\cD F_{\textup{indcons}}(X)$. Since $\textup{Gr}^i$ are exact functors, stability and that the Beilinson $t$-structure restricts to $\cD F_{\textup{cons}}(X)$ is clear. (a) follows from Corollary \ref{Beilinson_t_structure_Gr_i_Cor}, (a). Finally (b) follows from Corollary \ref{Beilinson_t_structure_Gr_i_Cor}, (b) since the inclusion of $\cD_{\textup{cons}}(X)$ into $\cD_{\textup{indcons}}(X)$ is $t$-exact for the perverse $t$-structure (see \S \ref{properties_ind_cons}, (c)). 
\end{proof}

\bibliographystyle{alpha}
\bibliography{arXiv.bib}

\begin{thebibliography}{HMMS22}

\bibitem[BBDG18]{BBDG18}
A.~Beilinson, J.~Bernstein, P.~Deligne, and O.~Gabber.
\newblock Faisceaux pervers.
\newblock {\em Soci{\'e}t{\'e} math{\'e}matique de France}, 100, 2018.

\bibitem[Bei87]{Bei87}
A.~A. Beilinson.
\newblock On the derived category of perverse sheaves.
\newblock In {\em {$K$}-theory, arithmetic and geometry ({M}oscow,
  1984--1986)}, volume 1289 of {\em Lecture Notes in Math.}, pages 27--41.
  Springer, Berlin, 1987.

\bibitem[Bei16]{B}
A.~Beilinson.
\newblock Constructible sheaves are holonomic.
\newblock {\em Selecta Math. (N.S.)}, 22(4):1797--1819, 2016.

\bibitem[BKV22]{BKV22}
Alexis Bouthier, David Kazhdan, and Yakov Varshavsky.
\newblock Perverse sheaves on infinite-dimensional stacks, and affine
  {S}pringer theory.
\newblock {\em Adv. Math.}, 408:Paper No. 108572, 132, 2022.

\bibitem[BMS19]{BMS19}
Bhargav Bhatt, Matthew Morrow, and Peter Scholze.
\newblock Topological {H}ochschild homology and integral {$p$}-adic {H}odge
  theory.
\newblock {\em Publ. Math. Inst. Hautes \'{E}tudes Sci.}, 129:199--310, 2019.

\bibitem[Bry86]{Br}
Jean-Luc Brylinski.
\newblock Transformations canoniques, dualit\'{e} projective, th\'{e}orie de
  {L}efschetz, transformations de {F}ourier et sommes trigonom\'{e}triques.
\newblock Number 140-141, pages 3--134, 251. 1986.
\newblock G\'{e}om\'{e}trie et analyse microlocales.

\bibitem[BS15]{BS15}
Bhargav Bhatt and Peter Scholze.
\newblock The pro-\'{e}tale topology for schemes.
\newblock {\em Ast\'{e}risque}, (369):99--201, 2015.

\bibitem[dCHM12]{dHM12}
Mark Andrea~A. de~Cataldo, Tam\'{a}s Hausel, and Luca Migliorini.
\newblock Topology of {H}itchin systems and {H}odge theory of character
  varieties: the case {$A_1$}.
\newblock {\em Ann. of Math. (2)}, 175(3):1329--1407, 2012.

\bibitem[dCMS22]{dMS22}
Mark de~Cataldo, Davesh Maulik, and Junliang Shen.
\newblock Hitchin fibrations, abelian surfaces, and the p= w conjecture.
\newblock {\em Journal of the American Mathematical Society}, 35(3):911--953,
  2022.

\bibitem[Del71]{Del71}
Pierre Deligne.
\newblock Th\'{e}orie de {H}odge. {II}.
\newblock {\em Inst. Hautes \'{E}tudes Sci. Publ. Math.}, (40):5--57, 1971.

\bibitem[Del77]{SGA4.5}
P.~Deligne.
\newblock {\em Cohomologie \'{e}tale}, volume 569 of {\em Lecture Notes in
  Mathematics}.
\newblock Springer-Verlag, Berlin, 1977.
\newblock S\'{e}minaire de g\'{e}om\'{e}trie alg\'{e}brique du Bois-Marie SGA
  $4\frac{1}{2}$.

\bibitem[Del80]{Del80}
Pierre Deligne.
\newblock La conjecture de {W}eil. {II}.
\newblock {\em Inst. Hautes \'{E}tudes Sci. Publ. Math.}, (52):137--252, 1980.

\bibitem[dM10]{deCM10}
M.~A. deCataldo and L.~Migliorini.
\newblock The perverse filtration and the lefschetz hyperplane theorem.
\newblock {\em Annals of Mathematics}, 171:1797--1819, 2010.

\bibitem[EK05]{EK05}
H\'{e}l\`ene Esnault and Nicholas~M. Katz.
\newblock Cohomological divisibility and point count divisibility.
\newblock {\em Compos. Math.}, 141(1):93--100, 2005.

\bibitem[EW22]{EW22}
H\'{e}l\`ene Esnault and Daqing Wan.
\newblock Divisibility of {F}robenius eigenvalues on {$\ell$}-adic cohomology.
\newblock {\em Proc. Indian Acad. Sci. Math. Sci.}, 132(2):Paper No. 60, 7,
  2022.

\bibitem[GP18]{GP18}
Owen Gwilliam and Dmitri Pavlov.
\newblock Enhancing the filtered derived category.
\newblock {\em J. Pure Appl. Algebra}, 222(11):3621--3674, 2018.

\bibitem[HMMS22]{HMMS22}
Tamas Hausel, Anton Mellit, Alexandre Minets, and Olivier Schiffmann.
\newblock $p=w$ via $h_2$, 2022.

\bibitem[HRS23]{HRS23}
Tamir Hemo, Timo Richarz, and Jakob Scholbach.
\newblock Constructible sheaves on schemes.
\newblock {\em Adv. Math.}, 429:Paper No. 109179, 46, 2023.

\bibitem[Kat73]{SGA7II}
N.~Katz.
\newblock Le niveau de la cohomologie des intersections compl{\`e}tes.
  {Appendice} par {P}. {Deligne}.
\newblock S{\'e}m. {G{\'e}om}. alg{\'e}brique {Bois}-{Marie} 1967--1969, {SGA}
  7 {II}, {Lect}. {Notes} {Math}. 340, {Exp}. {No}. {XXI}, 363-400 (1973).,
  1973.

\bibitem[Kat93]{Kat93}
Nicholas~M. Katz.
\newblock Affine cohomological transforms, perversity, and monodromy.
\newblock {\em J. Amer. Math. Soc.}, 6(1):149--222, 1993.

\bibitem[Lau03]{Lau}
G.~Laumon.
\newblock Transformation de fourier homog{\`e}ne.
\newblock {\em Bulletin de la Soci{\'e}t{\'e} Math{\'e}matique de France},
  131(4):527--551, 2003.

\bibitem[LurHA]{HA}
Jacob Lurie.
\newblock {\em Higher Algebra, https://www.math.ias.edu/~lurie/papers/HA.pdf}.
\newblock HA.

\bibitem[LurTT]{HTT}
Jacob Lurie.
\newblock {\em Higher topos theory}, volume 170 of {\em Annals of Mathematics
  Studies}.
\newblock Princeton University Press, Princeton, NJ, HTT.

\bibitem[MS22]{MS22}
Davesh Maulik and Junliang Shen.
\newblock The $p=w$ conjecture for $\mathrm{GL}_n$, 2022.

\bibitem[Nor02]{Nor02}
Madhav~V. Nori.
\newblock Constructible sheaves.
\newblock In {\em Algebra, arithmetic and geometry, {P}art {I}, {II} ({M}umbai,
  2000)}, volume~16 of {\em Tata Inst. Fund. Res. Stud. Math.}, pages 471--491.
  Tata Inst. Fund. Res., Bombay, 2002.

\bibitem[PS23]{DS23}
D.~Patel and K.~V. Shuddhodan.
\newblock Brylinski-radon transformation in charachteristic $p>0$.
\newblock {\em arXiv}, (2307.04156), 2023.

\bibitem[SGA71]{SGA6}
Th\'{e}orie des intersections et th\'{e}or\`eme de {R}iemann-{R}och.
\newblock Vol. 225:xii+700, 1971.
\newblock S\'{e}minaire de G\'{e}om\'{e}trie Alg\'{e}brique du Bois-Marie
  1966--1967 (SGA 6), Dirig\'{e} par P. Berthelot, A. Grothendieck et L.
  Illusie. Avec la collaboration de D. Ferrand, J. P. Jouanolou, O. Jussila, S.
  Kleiman, M. Raynaud et J. P. Serre.

\bibitem[SS16]{Schapira-Schneiders}
Pierre Schapira and Jean-Pierre Schneiders.
\newblock Derived categories of filtered objects.
\newblock {\em Ast\'{e}risque}, (383):103--120, 2016.

\end{thebibliography}

\end{document}